\documentclass[a4paper]{article}

\usepackage{amsfonts}
\usepackage{amsthm}
\usepackage{amsmath}
\usepackage{amssymb}
\usepackage{graphicx}

\usepackage[margin=2.54cm]{geometry}

\numberwithin{equation}{section}

\newtheorem{theorem}{Theorem}[section]
\newtheorem{lemma}[theorem]{Lemma}
\newtheorem{proposition}[theorem]{Proposition}
\newtheorem{con}[theorem]{Conjecture}
\newtheorem{cor}[theorem]{Corollary}
\theoremstyle{remark}
\newtheorem{remark}[theorem]{Remark}

\DeclareMathOperator{\Ric}{Ric}
\DeclareMathOperator{\Ein}{Ein}
\DeclareMathOperator{\tr}{tr}

\DeclareMathOperator{\sign}{sign}

\title{The prescribed cross curvature problem on the three-sphere}

\author{Timothy Buttsworth\thanks{School of Mathematics and Physics, The University of Queensland, St Lucia,~QLD 4072, Australia}~\thanks{Research supported by the Australian Government through the Australian Research Council's Discovery Projects funding scheme (DP180102185).} \\
\small{\texttt{t.buttsworth@uq.edu.au}}
\and Artem Pulemotov\footnotemark[1]~\footnotemark[2] \\
\small{\texttt{a.pulemotov@uq.edu.au}}}

\begin{document}

\maketitle

\begin{abstract}
The paper studies the problem of prescribing positive cross curvature on the three-dimensional sphere. We produce several existence results and an example of non-uniqueness, disproving a conjecture of Hamilton's.
\end{abstract}

\tableofcontents

\newpage 
\section{Introduction}

Let $(M,g)$ be a Riemannian manifold. Consider the Einstein tensor field of~$g$ given by $$\Ein(g)=\text{Ric}(g)-\frac{S(g)g}{2}$$
and the corresponding $(1,1)$-tensor field $\mathcal{E}(g)=\Ein(g)^{\sharp}$. Assuming that $\mathcal E(g)$ is invertible on~$M$, we define the \textit{cross curvature} $X(g)$ by the formula  
\begin{align*}
X(g)=\text{det}(\mathcal{E}(g))(\mathcal{E}^{-1})^{\flat}.
\end{align*}
One may allow the metric $g$ to depend on a time parameter $t\ge0$ and deform it in the direction of~$X(g)$. As a result, one obtains the evolution equation
\begin{equation}\label{XCF}
\frac{\partial g(t)}{\partial t}=X(g(t)), 
\end{equation}
called the \emph{cross curvature flow}. This equation was introduced by Chow--Hamilton in~\cite{ChowHam} in the context of uniformising negatively curved metrics. It has been investigated by many authors; see the survey~\cite{BryanSurvey}. The associated solitons were considered in~\cite{Glickenstein,HoShin}.

The effort to understand cross curvature leads to several natural questions: Which tensor fields arise as $X(g)$ for some metric $g$ on a given manifold? To what extent does $X(g)$ determine~$g$? If distinct metrics $g$ with the same $X(g)$ can be found, are they necessarily isometric? We use the umbrella term \emph{prescribed cross curvature problem} for all these questions. Essentially, the problem consists in establishing existence and uniqueness properties of solutions to the equation
\begin{align}\label{PXC_intro}
X(g)=Y,
\end{align}
where $Y$ is a given tensor field. It is to be expected that research into~\eqref{PXC_intro} will yield a better understanding of the evolution equation~\eqref{XCF} as well as $X(g)$ itself. For instance, DeTurck's work on metrics with prescribed Ricci curvature has led to the discovery of the famous DeTurck trick for the Ricci flow in~\cite{DDTtrick}. Many of the techniques that are effective for the analysis of~\eqref{PXC_intro} are likely to be effective on the soliton equation associated with~\eqref{XCF}.

Assume that $M$ is three-dimensional. Instead of defining the cross curvature as above, one may think of it as a $(1,2)$-tensor field on~$M$. We explain the details of this approach in Section~\ref{XC3D}. It is equivalent to ours; see Proposition~\ref{prop_12to02}. The first systematic study of the prescribed cross curvature problem was carried out for $(1,2)$-tensor fields by Gkigkitzis in his Ph.D. thesis~\cite{IG08} written at Columbia. He proved several useful properties of $X(g)$ and made several interesting observations related to equation~\eqref{PXC_intro}.

The existing literature on the cross curvature focuses primarily on the case where $M$ is three-dimensional and closed. If one assumes, in addition, that $M$ is simply-connected, then it must be the sphere~$\mathbb S^3$. The following conjecture (more precisely, its equivalent form for $(1,2)$-tensor fields) was stated in~\cite{IG08} and attributed to Hamilton. However, the only case where it was settled was one with $\mathbb S^3$ viewed as $SU(2)$ and the analysis restricted to left-invariant metrics.

\begin{con}\label{mainconjecture}
Given a positive-definite symmetric (0,2)-tensor field $Y$ on~$\mathbb{S}^3$, there exists a unique Riemannian metric $g$ solving (\ref{PXC_intro}).
\end{con}

In the present paper, we obtain a series of results concerning the prescribed cross curvature problem on~$\mathbb S^3$. Our theorems provide a large body of evidence to support the existence portion of Conjecture~\ref{mainconjecture}. Perhaps surprisingly, they demonstrate that uniqueness fails in general. In Section~\ref{local}, we prove existence and regularity of solutions to~\eqref{PXC_intro} in a neighbourhood of a point on~$\mathbb S^3$. These results are in the spirit of the early works on the prescribed Ricci curvature problem by DeTurck and DeTurck--Kazdan; see~\cite{DDT81,DeTurckKazdan}. In Section~\ref{sec_round}, we show that a metric $g$ satisfying~\eqref{PXC_intro} can be found if $Y$ is close to the round metric in an appropriate Sobolev norm. The counterpart of this result for (1,2)-tensors was conjectured but not proven in~\cite{IG08}. In Sections~\ref{sec_so2_exist} and~\ref{SectionexistenceSO3}, we establish existence under the assumption that $Y$ is invariant under the action of $SO(2)\times SO(2)$ and $SO(3)$, respectively. (For the former symmetry, we also require that $Y$ be ``diagonal".) These are the only effective cohomogeneity one actions on~$\mathbb S^3$; see~\cite{Neumann68}. In Section~\ref{sec_nonuniq}, we demonstrate that the uniqueness portion of Conjecture~\ref{mainconjecture} fails for some $SO(3)$-invariant choices of~$Y$. Moreover, we construct \emph{non-isometric} metrics with the same cross curvature. While Gkigkitzis discussed tackling~\eqref{PXC_intro} under some cohomogeneity one symmetry assumptions by variational methods in~\cite{IG08}, he did not prove any theorems or state any specific conjectures. Finally, in Section~\ref{SU2}, we address the prescribed cross curvature problem for left-invariant metrics on~$SU(2)$. In this setting, Conjecture~\ref{mainconjecture} was verified in~\cite{IG08}. We re-state the result and establish some dynamical properties of the map~$g\mapsto X(g)$.

Throughout the paper, the terms \emph{metric} and \emph{tensor field} refer to smooth metrics and tensor fields unless indicated otherwise.

\section{Cross curvature on the three-sphere}\label{XC3D}

Let $M$ be a smooth connected oriented manifold. Given a Riemannian metric $g$ on~$M$, we use the notation $\Ric(g)$ and $S(g)$ for its Ricci and scalar curvature. Consider the Einstein (1,1)-tensor field
\begin{align*}
    \mathcal E(g)=\Big(\Ric(g)-\frac{S(g)g}{2}\Big)^\sharp,
\end{align*}
where $\sharp$ is one of the musical isomorphisms. Assuming $\mathcal E(g)$ is invertible, define
\begin{align}\label{DefnXC'}
    X(g)(\cdot,\cdot)=\det(\mathcal{E}(g))(\mathcal{E}(g)^{-1})^{\flat}(\cdot,\cdot)=\det(\mathcal{E}(g))\,g(\mathcal{E}(g)^{-1}\cdot,\cdot).
\end{align}
We call $X(g)$ the \emph{cross curvature (0,2)-tensor} of~$g$.

Suppose $M$ has dimension~3. In this case, one may approach the concept of the cross curvature in a different way. Specifically, let $\mu(g)$ and $\nu(g)$ be the volume $3$-form of~$g$ and the $(3,0)$-tensor field obtained from $\mu(g)$ by raising the indices. We define the \emph{cross curvature (1,2)-tensor} $T(g)$ by the formula
\begin{align*}
g(T(g)(e_i,e_j),e_m)=\tfrac12\sum_{k,l=1}^{3}\textrm{Rm}(e_k,e_l,e_i,e_j)\nu\big(e_k^{\flat},e_l^{\flat},e_m^{\flat}\big),
\end{align*}
where Rm is the Riemann curvature and $\{e_1,e_2,e_3\}$ is an orthonormal positively oriented basis of the tangent space. For all $A$ and $B$, the equality $T(g)(A,B)=-T(g)(B,A)$ holds, and the traces of $T(g)(A,\cdot)$ and $T(g)(\cdot,B)$ are both~0. Following~\cite{IG08}, we call a $(1,2)$-tensor field $\tilde{Y}$ that possesses these two properties a \textit{cross tensor field}. One may interpret $\tilde Y$ as a map from $TM\wedge TM$ to~$TM$. If $\{A,B,\tilde{Y}(A,B)\}$ is a positively oriented basis for all linearly independent $A$ and $B$, then $\tilde{Y}$ is said to be positive.

Throughout the rest of this paper, we assume that $M$ is the sphere~$\mathbb S^3$. Let us explain the relation between $X(g)$ and~$T(g)$. Take a positive-definite symmetric $(0,2)$-tensor field $Y$ on~$\mathbb{S}^3$. Let $\star$ be the Hodge star operator associated with~$Y$. Define the anti-symmetric $(1,2)$-tensor field $\mathcal F(Y)$ by setting
$$
\mathcal F(Y)(A,B)=\star\,(A\wedge B).
$$
The relation between $X(g)$ and~$T(g)$ is given by the following result.

\begin{proposition}\label{prop_12to02} 
The map $\mathcal F$ is a bijection between the space of positive-definite symmetric (0,2)-tensor fields and the space of positive cross tensor fields on~$\mathbb{S}^3$. Moreover, given a metric $g$ with $X(g)>0$, the equality $X(g)=Y$ holds if and only if $T(g)=\mathcal F(Y)$.
\end{proposition}

\begin{proof}
We begin by proving that $\mathcal F$ is surjective. Choose a positive cross tensor field $\tilde Y$ and a background metric $g_0$ on~$\mathbb S^3$ with Hodge star operator~$*$\,. Consider a linear map $L:T\mathbb{S}^3\to T\mathbb{S}^3$ defined by $L(u)=\tilde{Y}(*u)$. The anti-symmetry and traceless properties of $\tilde{Y}$ imply that $L$ is self-adjoint with respect to~$g_0$. Therefore, for every $p\in\mathbb S^3$, one can find a $g_0$-orthonormal positively oriented basis of $T_p\mathbb S^3$ in which~$L$ is diagonal. Rescaling its vectors if necessary, one obtains a new basis, $\{V_1,V_2,V_3\}$, such that $\tilde{Y}(V_i,V_j)=V_k$ for each even permutation $(i,j,k)$ of $\{1,2,3\}$. Define $Y(p)$ to be the (0,2)-tensor at $p$ given by the identity matrix in~$\{V_1,V_2,V_3\}$. Clearly, $\mathcal F(Y)=\tilde Y$.

Next, we show that $\mathcal F$ is injective. Suppose that $\mathcal F(Y^1)=\mathcal F(Y^2)=\tilde Y$. Given~$p\in\mathbb S^3$, choose a positively oriented $Y^l$-orthonormal basis $\{V_1^l,V_2^l,V_3^l\}$ of the space $T_p\mathbb S^3$ for each $l=1,2$. Denote by $\star_l$ the Hodge star operator associated with~$Y^l$. We have
\begin{align*}
\star_1(V_i^2\wedge V_j^2)=\mathcal F(Y^1)(V_i^2,V_j^2)=\mathcal F(Y^2)(V_i^2,V_j^2)=\star_2(V_i^2\wedge V_j^2)=V_k^2,
\end{align*}
provided that $(i,j,k)$ is an even permutation of $\{1,2,3\}$. Consequently, the basis $\{V_1^2,V_2^2,V_3^2\}$ must be $Y^1$-orthonormal, which means $Y^1=Y^2$.

Let us prove that $X(g)=Y$ if and only if $T(g)=\mathcal F(Y)$. Given a Riemannian metric $g$ on $\mathbb{S}^3$ and a point~$p\in\mathbb S^3$, choose a positively oriented $g$-orthonormal basis $\{e_1,e_2,e_3\}$. Without loss of generality, assume that $e_2\wedge e_3$, $e_3\wedge e_1$ and $e_1\wedge e_2$ are eigenvectors of the curvature operator of $g$ with eigenvalues $\lambda_1$, $\lambda_2$ and $\lambda_3$, respectively. Then $\{e_1,e_2,e_3\}$ diagonalises~$\mathcal E(g)$ and $X(g)$. In fact, $\mathcal E(g)e_i=-\lambda_i e_i$ and $X(g)(e_i,e_i)=\lambda_j\lambda_k$ for every permutation $(i,j,k)$ of $\{1,2,3\}$.

If $X(g)=Y$, then 
$$
\frac{e_k}{\sqrt{\lambda_i\lambda_j}}=\star\,\bigg(\frac{e_i}{\sqrt{\lambda_j\lambda_k}}\wedge\frac{e_j}{\sqrt{\lambda_i\lambda_k}}\bigg),
$$
provided $(i,j,k)$ is even. As a result,
\begin{align*}
T(g)(e_i,e_j)&=\sum_{m=1}^3g(T(e_i,e_j),e_m)e_m
\\
&=\lambda_ke_k=\lambda_k\sqrt{\lambda_i\lambda_j}\star\bigg(\frac{e_i}{\sqrt{\lambda_j\lambda_k}}\wedge\frac{e_j}{\sqrt{\lambda_i\lambda_k}}\bigg)=\sign(\lambda_k)\mathcal F(Y)(e_i,e_j).
\end{align*}
Since $Y$ is positive-definite, $\sign(\lambda_i)$ is the same for all~$i$. Moreover, it does not depend on the choice of the point $p\in\mathbb S^3$ by the continuity of $\mathcal{E}(g)$ and the fact that $\mathcal E(g)$ is invertible across~$\mathbb S^3$. If $\sign(\lambda_i)=-1$, then $g$ is a metric with negative sectional curvature. However, the existence of such a metric is ruled out by the Cartan--Hadamard theorem. We conclude that $T(g)(e_i,e_j)=\mathcal{F}(Y)(e_i,e_j)$.

Conversely, suppose that $T(g)=\mathcal F(Y)$. Then
$$
\star\bigg(\frac{e_i}{\sqrt{\lambda_j\lambda_k}}\wedge\frac{e_j}{\sqrt{\lambda_i\lambda_k}}\bigg)=\mathcal{F}(Y)\bigg(\frac{e_i}{\sqrt{\lambda_j\lambda_k}},\frac{e_j}{\sqrt{\lambda_i\lambda_k}}\bigg)=T(g)\bigg(\frac{e_i}{\sqrt{\lambda_j\lambda_k}},\frac{e_j}{\sqrt{\lambda_i\lambda_k}}\bigg)=\frac{e_k}{\sqrt{\lambda_i\lambda_j}}
$$
for every even permutation $(i,j,k)$ of $\{1,2,3\}$. This implies that the vectors
$$
V_1=\frac{e_1}{\sqrt{\lambda_2\lambda_3}},\qquad V_2=\frac{e_2}{\sqrt{\lambda_1\lambda_3}}\qquad\mbox{and}\qquad V_3=\frac{e_3}{\sqrt{\lambda_1\lambda_2}}
$$
form a $Y$-orthonormal basis. Consequently,
\begin{align*}
X(g)(e_i,e_l)&=\lambda_j\lambda_k Y(V_i,V_l)=Y(e_i,e_l).
\end{align*}
Here, as above, $(i,j,k)$ is an even permutation of $\{1,2,3\}$, and $l=1,2,3$.

\end{proof}

The following conjecture appeared in~\cite{IG08}, attributed to Hamilton.

\begin{con}\label{rc12}
Let $\tilde{Y}$ be a (1,2)-tensor field on $\mathbb{S}^3$ such that $\tilde{Y}(p)$ is a positive cross tensor for every~$p\in\mathbb S^3$. There exists a unique Riemannian metric $g$ on $\mathbb S^3$ with~$T(g)=\tilde{Y}$. 
\end{con}

Proposition~\ref{prop_12to02} implies that Conjecture~\ref{rc12} is equivalent to Conjecture~\ref{mainconjecture}. We will prove a series of results that support the existence portion of these conjectures. Afterwards, we will demonstrate non-uniqueness.

\section{Local properties}\label{local}

In this section, we focus on the local variant of the prescribed cross curvature problem. Choose a point~$p\in M$. Our first goal is to solve the equation $X(g)=Y$ in a neighbourhood of $p$ provided that $Y$ is positive-definite. Next, we will prove that $g$ must lie in the H\"older class $C^{k,\alpha}$ with $\alpha\in(0,1)$ if $g$ is $C^3$ and $Y$ lies in~$C^{k,\alpha}$.

Choose a local coordinate system on $\mathbb S^3$ centered at~$p$. Using the standard notation for the components of tensors, we obtain
\begin{align*}
X(g)_{ij}=\det(\mathcal E(g))g_{ik}(\mathcal E(g)^{-1})_j^k.
\end{align*}
It will be convenient for us to fix a background metric $g_0$ on~$\mathbb S^3$. Linearising $X(g)$ near this metric, we find
\begin{align}\label{LXCCO}
     X(g_0+h)_{ij}-X(g_0)_{ij}=\tfrac12g_0^{km}\mathcal E(g_0)_m^l\bigg(
     \frac{\partial^2 h_{ki}}{\partial x_l\partial x_j}&+\frac{\partial^2 h_{kj}}{\partial x_l\partial x_i}-\frac{\partial^2 h_{kl}}{\partial x_i\partial x_j}-\frac{\partial^2 h_{ij}}{\partial x_l\partial x_k}\bigg) \notag
\\
     &+L(h)+G(h,Dh,D^2h);
\end{align}
see~\cite[Proof of Lemma~4]{ChowHam}. In this formula, $Dh$ and $D^2h$ are the arrays of first and second derivatives of the components of $h$ in the local coordinates, and $L$ is a first-order linear differential operator. The smooth map $G$ satisfies
\begin{align*}
G(h,Dh,D^2h)=O\big(|h|^2+|Dh|^2+|D^2h|^2\big),\qquad |h|^2+|Dh|^2+|D^2h|^2\to0,
\end{align*}
where
$$
|h|^2=\sum_{i,j=1}^3h_{ij}^2,\qquad |Dh|^2=\sum_{i,j=1}^3\sum_{k=1}^3\bigg|\frac{\partial h_{ij}}{\partial x_k}\bigg|^2,\qquad |D^2h|^2=\sum_{i,j=1}^3\sum_{k,l=1}^3\bigg|\frac{\partial^2 h_{ij}}{\partial x_k\partial x_l}\bigg|^2.
$$
The main difficulty in proving local existence and regularity comes from the fact that, as we see from~\eqref{LXCCO}, the prescribed cross curvature equation is not elliptic.

\subsection{Existence}

Our first theorem mirrors DeTurck's result on the local existence of metrics with prescribed Ricci curvature; see~\cite{DDT81} and~\cite[Chapter~5]{Besse}.

\begin{theorem}\label{thm_local}
Let $Y$ be a positive-definite symmetric (0,2)-tensor field on~$\mathbb S^3$. There exists a Riemannian metric $g$ on~$\mathbb S^3$ such that $X(g)=Y$ in a neighbourhood of~$p$.
\end{theorem}

\begin{proof}
Given a symmetric $(0,2)$-tensor field~$h$ close to 0 in the $C^0$ norm, define a vector field $V(h)$ in a neighbourhood of~$p$ by setting
$$V^i(h)=(Y^{-1})^{ij}(g_0+h)^{kl}\big(\overline{\nabla}_l h_{kj}-\tfrac{1}{2}\overline{\nabla}_j h_{kl}\big),
$$
where $(Y^{-1})^{ij}$ are the entries of the inverse of the matrix~$(Y_{ij})_{i,j=1}^3$ and $\overline{\nabla}$ is the Levi-Civita connection of~$g_0$. Consider the map $\phi_h(x)=\text{exp}_x(V(h))$. Clearly, it is a diffeomorphism between neighbourhoods of $p$ if $h$ vanishes at $p$ together with its derivatives. We will now prove that the equation $E(h)=0$, where $E(h)=X(g_0+h)-\phi_h^*(Y)$, can be solved for such~$h$.

According to~\eqref{LXCCO} and the well-known formula for the linearisation of a diffeomorphism (see~\cite[Section~5.C]{Besse}),
\begin{align*}
    E(h)&=\tfrac{1}{2}g_0^{km}\mathcal E(g_0)_m^l\bigg(
     \frac{\partial^2 h_{ki}}{\partial x_l\partial x_j}+\frac{\partial^2 h_{kj}}{\partial x_l\partial x_i}-\frac{\partial^2 h_{kl}}{\partial x_i\partial x_j}-\frac{\partial^2 h_{ij}}{\partial x_l\partial x_k}\bigg)
     \\
     &\hphantom{=}~-\tfrac{1}{2}g_0^{kl}\bigg(
     \frac{\partial^2 h_{ki}}{\partial x_l\partial x_j}+\frac{\partial^2 h_{kj}}{\partial x_l\partial x_i}-\frac{\partial^2 h_{kl}}{\partial x_i\partial x_j}\bigg)+L(h)+ O(|h|^2+\left|Dh\right|^2+|D^2 h|^2),
\end{align*}
where $L$ is a linear first-order differential operator. Without loss of generality, assume that $g_0$, $Y$ and $\mathcal{E}(g_0)$ are given by identity matrices at $p$ in our chosen local coordinates; cf.~\cite[Section~5.A]{Besse}. It is then easy to see that the equation $E(h)=0$ is elliptic at $h=0$ near~$p$. Clearly, $X(g_0)=Y$ at~$p$. The existence of a local solution to $E(h)=0$ follows from Theorem~45 in~\cite[Appendix~K]{Besse}. Pulling back by $\phi_h^{-1}$, we obtain $g$ such that $X(g)=Y$.
\end{proof}

\begin{remark}
Because Theorem~\ref{thm_local} is a result of local nature, it holds on any three-dimensional manifold~$M$, not just on~$\mathbb S^3$. Moreover, it seems to admit a generalisation to the case where ${\dim M>3}$. We will not pursue such a generalisation in this paper.
\end{remark}

\subsection{Regularity}

Our next result shows that differentiability of the cross curvature implies that of the metric. We will use it in Section~4 when we study the existence of global solutions to~$X(g)=Y$.

\begin{theorem}\label{regularityxcurvature}
Suppose that $g$ is a $C^3$-differentiable Riemannian metric on $\mathbb S^3$ with positive-definite cross curvature (0,2)-tensor~$X(g)=Y$. If $Y$ lies in the H\"older class~$C^{k,\alpha}$ for some $k\in[3,\infty)\cap\mathbb N$ and $\alpha\in (0,1)$, then so does~$g$.
\end{theorem}

\begin{proof}
Choose a $g$-normal coordinate system centered at $p\in\mathbb S^3$. Let the background metric $g_0$ be such that its component $g_{0ij}$ near $p$ is the Taylor polynomial of degree 2 of $g_{ij}$ at $p$ for all $i,j=1,2,3$. Then $g=g_0+h$ for some (0,2)-tensor field $h$ vanishing at~$p$ together with its first and second derivatives. Denote by $\det(Y)$ the determinant of the matrix~$(Y_{ij})_{i,j=1}^3$. Formula~\eqref{DefnXC'} shows that
$$
\det(\mathcal E(g_0))=\pm\sqrt{\det(Y)}\qquad\mbox{and}\qquad g_0^{km}\mathcal E(g_0)_m^l=\pm\sqrt{\det(Y)}(Y^{-1})^{kl}
$$
at~$p$. We assume for the rest of the proof that the square roots on the right-hand sides appear with pluses. The argument is similar in the other case.

The second Bianchi identity implies
\begin{align*}
(Y^{-1})^{ij}\nabla_i Y_{jk}=\tfrac{1}{2}(Y^{-1})^{ij}\nabla_k Y_{ij},
\end{align*}
where $\nabla$ the Levi--Civita connection of~$g$; see~\cite[Lemma~1]{ChowHam}. Using the standard notation $\Gamma_{ij}^k$ for the Christoffel symbols of~$g$, we find
\begin{align*}
(Y^{-1})^{ij}\Gamma_{ij}^l&=(Y^{-1})^{ij}(Y^{-1})^{kl}\Gamma_{ij}^{m}Y_{mk}
\\
&=(Y^{-1})^{ij}(Y^{-1})^{kl}(\Gamma_{ij}^{m}Y_{mk}+\Gamma_{ik}^{m}Y_{jm})-\tfrac{1}{2}(Y^{-1})^{ij}(Y^{-1})^{kl}(\Gamma_{ki}^{m}Y_{jm}+\Gamma_{ki}^{m}Y_{mj})\\
&=(Y^{-1})^{ij}(Y^{-1})^{kl}\Big(\frac{\partial}{\partial x_i}Y_{jk}-\nabla_i Y_{jk}\Big)-\tfrac{1}{2}(Y^{-1})^{ij}(Y^{-1})^{kl}\Big(\frac{\partial}{\partial x_k}Y_{ij}-\nabla_k Y_{ij}\Big)\\
&=F_l(g,Y,DY),
\end{align*}
where $F_l$ is a smooth map for every $l=1,2,3$. Differentiation with respect to $x_r$ yields
\begin{align*}
(Y^{-1})^{ij}g^{lm}\left(2\frac{\partial^2 h_{im}}{\partial x_j\partial x_r}-\frac{\partial^2 h_{ij}}{\partial x_m\partial x_r}\right)=\hat{F}_{lr}(h,Dh,Y,DY,D^2 Y).
\end{align*}
Here, $\hat F_{lr}$ is a smooth map for all $l,r=1,2,3$. 
In light of~\eqref{LXCCO},
\begin{align}\label{big_Bianchi_eq}
g_0^{km}&\mathcal E(g_0)_m^l 
\bigg(\frac{\partial^2 h_{ki}}{\partial x_l\partial x_j}+\frac{\partial^2 h_{kj}}{\partial x_l\partial x_i}-\frac{\partial^2 h_{kl}}{\partial x_i\partial x_j}-\frac{\partial^2 h_{ij}}{\partial x_l\partial x_k}\bigg) \notag
\\
&\hphantom{=}~-
\sqrt{\det(Y)}(Y^{-1})^{kl}g^{mi}\left(\frac{\partial^2 h_{km}}{\partial x_l\partial x_j}-\frac{\partial^2 h_{kl}}{2\partial x_m\partial x_j}\right) \notag
\\
&\hphantom{=}~-
\sqrt{\det(Y)}
(Y^{-1})^{kl}g^{mj}\left(\frac{\partial^2 h_{km}}{\partial x_l\partial x_i}-\frac{\partial^2 h_{kl}}{2\partial x_m\partial x_i}\right) \notag
\\
&=2Y-2X(g_0)-2L(h)-2G(h,Dh,D^2h)-\tfrac12\sqrt{\det(Y)}(\hat{F}_{ij}+\hat{F}_{ji})(h,Dh,Y,DY,D^2 Y).
\end{align}
At the point~$p$, this becomes
\begin{align*}
\sqrt{\det(Y)}(Y^{-1})^{kl}\frac{\partial^2 h_{ij}}{\partial x_l\partial x_k}&=-2G(0,0,0)
\\
&\hphantom{=}~-\tfrac12\sqrt{\det(Y)}(\hat{F}_{ij}+\hat{F}_{ji})(0,0,Y,DY,D^2Y).
\end{align*}
Consequently, in a neighbourhood of~$p$,~\eqref{big_Bianchi_eq} is an elliptic system of equations for~$(h_{ij})_{i,j=1}^3$. We differentiate each of these equations in $x_q$ for $q=1,2,3$. As a result, we obtain a quasilinear elliptic system for $\big(\frac{\partial h_{ij}}{\partial x_q}\big)_{i,j=1}^3$. Well-known regularity properties of such systems (see, e.g.,~\cite[Section~8.5]{LS68}) and a standard bootstrapping argument imply the assertion of the theorem.
\end{proof}

\begin{remark}
Theorem~\ref{regularityxcurvature} is a local result, just like Theorem~\ref{thm_local}. It admits a generalisation to an arbitrary manifold of dimension~3 and possibly higher.
\end{remark}

\section{Existence near the round metric}\label{sec_round}

In order to continue our study of metrics with prescribed cross curvature, we need to compute the linearisation of $X(g)$ at a round metric~$g_0$. It will be convenient for us to assume that $\Ric(g_0)=2g_0$. However, the scaling properties of the cross curvature imply that this assumption is not essential. In what follows, we use the standard notation $S^2T^*\mathbb S^3$ for the bundle of symmetric (0,2)-tensors on~$\mathbb S^3$ and $\Gamma(S^2T^*\mathbb S^3)$ for the space of its smooth sections.

\begin{theorem}\label{crosscurvaturelinearisation}
The linearisation $X'_{g_0}:\Gamma(S^2T^*\mathbb S^3)\to\Gamma(S^2T^*\mathbb S^3)$ at the round metric $g_0$ of the map $g\mapsto X(g)$ is given by the formula
\begin{align*}
    X'_{g_0}h=-\tfrac{1}{2}\Delta_Lh-\delta^*\delta G(h)-h,
\end{align*}
where $\Delta_L$ is the Lichnerowicz Laplacian, $\delta$ is the divergence operator, $\delta^*$ is the symmetrised covariant derivative, and $G$ is the gravitation tensor, all computed with respect to~$g_0$. 
\end{theorem}

\begin{proof}
Let us begin by expressing $X'_{g_0}(h)$ in terms of the linearisations of the maps $g\mapsto\mathcal E(g)$ and $g\mapsto\det(\mathcal E(g))$ at~$g_0$. The definition of $X(g)$ implies
\begin{align*}
    X'_{g_0}(h)=(\det(\mathcal{E}))_{g_0}'(h) g_0(\mathcal{E}(g_0)^{-1}\cdot,\cdot)&+\det(\mathcal{E}(g_0))h(\mathcal{E}(g_0)^{-1}\cdot,\cdot) \\ &-\det(\mathcal{E}(g_0))g_0(\mathcal{E}(g_0)^{-1}\mathcal{E}'_{g_0}(h)\mathcal{E}(g_0)^{-1}\cdot,\cdot).
\end{align*}
Jacobi's formula for the derivative of the determinant yields
\begin{align*}
    (\det(\mathcal{E}))_{g_0}'(h)= \det(\mathcal{E}(g_0))\tr(\mathcal{E}(g_0)^{-1}\mathcal{E}_{g_0}'(h)).
\end{align*}
Since $\text{Ric}(g_0)=2g_0$, the Einstein tensor $\mathcal{E}(g_0)$ is the negative of the identity. Consequently,
\begin{align*}
(\det(\mathcal{E}))_{g_0}'(h)= \tr\mathcal{E}_{g_0}'(h).
\end{align*}
Denoting $\mathcal{R}(g)=\text{Ric}(g)^{\sharp}$ and differentiating the equality $\text{Ric}(g)(\cdot,\cdot)=g(\mathcal{R}(g)\cdot,\cdot)$, we obtain
\begin{align*}
    \text{Ric}'_{g_0}(h)&=h(\mathcal{R}(g_0)\cdot,\cdot)+g_0(\mathcal{R}'_{g_0}(h)\cdot,\cdot), \\
     S_{g_0}'(h)&=\tr_{g_0}\text{Ric}'_{g_0}(h)-\tr_{g_0}h(\mathcal{R}(g_0)\cdot,\cdot).
\end{align*}
Therefore, 
\begin{align*}
    g_0(\mathcal{E}'_{g_0}(h)\cdot,\cdot)=
     \text{Ric}'_{g_0}(h)- h(\mathcal{R}(g_0)\cdot,\cdot)-\frac{\tr_{g_0}\text{Ric}'_{g_0}(h)}{2}g_0+\frac{\tr_{g_0}h(\mathcal{R}(g_0)\cdot,\cdot)}{2}g_0.
\end{align*}
The equality $\text{Ric}(g_0)=2g_0$ implies that $\mathcal{R}(g_0)$ is twice the identity, which means
\begin{align*}
    g_0(\mathcal{E}'_{g_0}(h)\cdot,\cdot)=
     \text{Ric}'_{g_0}(h)-\frac{\tr_{g_0}\text{Ric}'_{g_0}(h)}{2}g_0-2h+(\tr_{g_0}h)g_0.
\end{align*}
From here, we find
\begin{align*}
    X'_{g_0}(h)&=-((\det\mathcal{E})_{g_0}'(h)) g_0+g_0(\mathcal{E}'_{g_0}(h)\cdot,\cdot)+h
    \\ &=-(\tr_{g_0}g_0(\mathcal E_{g_0}'(h)\cdot,\cdot))g_0+g_0(\mathcal{E}'_{g_0}(h)\cdot,\cdot)+h=\Ric'_{g_0}(h)-h.
\end{align*}
Using the well-known formula
\begin{align*}
\text{Ric}'_{g_0}(h)=-\frac{1}{2}\Delta_Lh-\delta^*\delta G(h)
\end{align*}
(see, e.g.,~\cite[Proposition~2.3.7]{Topping}) completes the proof.
\end{proof}

Next, we state our first global existence result. Essentially, it shows that the prescribed cross curvature problem admits a solution if the cross curvature candidate lies near the round metric~$g_0$. Given a positive integer $k$, define
\begin{align*}
\langle h_1,h_2\rangle_k=\bigg(\int_{\mathbb S^3}g_0\big((C(k)I-\Delta_{L})^{\frac k2}h_1,(C(k)I-\Delta_{L})^{\frac k2}h_2\big)^2\,d\mu(g_0)\bigg)^{\frac12}
\end{align*}
for $h_1,h_2\in\Gamma(S^2T^*\mathbb S^3)$. Here $I$ is the identity operator. The constant $C(k)>0$ is chosen so that $\langle\cdot,\cdot\rangle_k$ is an inner product on~$\Gamma(S^2T^*\mathbb S^3)$ with the corresponding norm equivalent to the usual Sobolev $H^k$-norm; see~\cite[Lemma 2.1]{ButtsworthHallgren}. The Lichnerowicz Laplacian $\Delta_{L}$ is taken with respect to the metric~$g_0$. Given a tensor bundle $V$ over~$\mathbb S^3$, denote by $H^k(V)$ the completion of $\Gamma(V)$ in the Sobolev $H^k$-norm induced by~$g_0$. 

\begin{theorem}
Consider a symmetric (0,2)-tensor field $Y$ on~$\mathbb S^3$. Given an integer $k\ge5$, if $$\|Y-g_0\|_{H^k}<\epsilon(k)$$ for a sufficiently small $\epsilon(k)>0$, then there exists a Riemannian metric~$g$ such that $X(g)=Y$.
\end{theorem}

An analogue of this theorem for the cross curvature $(1,2)$-tensor was conjectured, though not proven, in~\cite{IG08}.

\begin{proof}
Denote by $\mathfrak X$ and $\mathfrak Y$ the completions of $\Gamma(S^2T^*\mathbb S^3)$ with respect to $\langle\cdot,\cdot\rangle_{k+2}$ and $\langle\cdot,\cdot\rangle_k$, respectively. If $l$ is an integer greater than~1, the Sobolev space $H^l(\mathbb{R}^3)$ is a Banach algebra. Consequently, a contraction of two tensor fields from $H^{k+m}(S^2T^*\mathbb{S}^3)$ lies in $H^{k+m}(S^2T^*\mathbb{S}^3)$ for each $m\in\mathbb N\cup\{0\}$. Also, $H^l(\mathbb{R}^3)$ is embedded continuously into the H\"older space~$C^{l-2,\frac12}(\mathbb{R}^3)$. These observations imply that, for some neighbourhood $U$ of 0 in~$\mathfrak X$, the map $F:U\to \mathfrak Y$ given by
$$
F(h)=X(g_0+h),\qquad h\in U,
$$
is smooth. By Theorem~\ref{crosscurvaturelinearisation}, the Fr\'echet derivative of $F$ at $h=0$ satisfies
\begin{equation}\label{FDXC}
F'_0f=-\tfrac{1}{2}\Delta_Lf-\delta^*\delta G(f)-f,\qquad f\in\mathfrak X.
\end{equation}
Our plan is to show that $F$ is invertible near $h=0$. To do so, we need to split $\mathfrak X$ and $\mathfrak Y$ into orthogonal sums.

The space $H^{k}(\ker\delta)$ is closed in~$\mathfrak Y$. The image $\delta^* (H^{k+1}(T^*M))$ is orthogonal to~$H^{k}(\ker\delta)$ in~$\mathfrak Y$. Indeed, if $h\in H^{k}(\text{ker}\delta)$ is smooth and $v$ lies in~$H^{k+1}(T^*M)$, then
\begin{align*}
\langle h,\delta^*v\rangle_k&=\int_{\mathbb{S}^3} g_0((C(k)-\Delta_L)^kh, \delta^*v)d\mu(g_0)=\int_{\mathbb{S}^3}g_0(\delta(C(k)-\Delta_L)^kh, v)d\mu(g_0).
\end{align*}
Since $\ker\delta$ is invariant under $\Delta_L$, this must equal~0. Using~\cite[Theorem~4.1]{BE}, we conclude that $\mathfrak Y$ splits into the orthogonal sum
\begin{align*}
\mathfrak Y&=H^k(\ker\delta)\oplus\delta^*(H^{k+1}(T^*M)).
\end{align*}
Analogous arguments show that
\begin{align}\label{Xsplit}
\mathfrak X=H^{k+2}(\ker\delta)\oplus\delta^*(H^{k+3}(T^*M)).
\end{align}
To prove the invertibility of $F$, we will apply the inverse function theorem to $F|_{H^{k+2}(\ker\delta)}$ and then take advantage of the diffeomorphism-invariance of the cross curvature.

Formulas~\eqref{FDXC} and~\eqref{Xsplit} imply that any $f\in H^{k+2}(\ker\delta)$ satisfying $F'_0f=0$ is either~0 or an eigenvector of $\Delta_L$ with eigenvalue~$-2$. According to~\cite[Theorem~3.2]{Boucetta99}, the latter is impossible. This means $F_0'|_{H^{k+2}(\ker\delta)}$ has trivial kernel. Moreover, the image $F_0'(H^{k+2}(\ker\delta))$ is closed in~$\mathfrak Y$. To see this, take a sequence $(f_n)_{n=1}^{\infty}\subset H^{k+2}(\ker \delta)$ such that 
$$\lim_{n\to\infty}F_0'(f_n)=y\in\mathfrak Y.$$
Projecting onto $H^k(\ker\delta)$, we conclude that $-\frac{1}{2}\Delta_Lf_n-f_n$ converges in~$\mathfrak Y$. The ellipticity of the operator $-\frac{1}{2}\Delta_L-I$ then implies the existence of $f\in H^{k+2}(\ker\delta)$ such that $\lim_{n\to \infty}f_n=f$ in $H^{k+2}(\ker\delta)$; cf.~\cite[Theorem~3.9]{Cantor}. It becomes clear that $F_0'(f)=y$. Thus, $F_0'(H^{k+2}(\ker\delta))$ is closed. By the inverse function theorem, there exists a neighbourhood $V$ of $0$ in $H^{k+2}(\ker\delta)$ such that the image $F(V)$ is a $C^1$-differentiable Banach submanifold of~$\mathfrak Y$. The tangent space of this submanifold at $g_0=F(0)$ is $F_0'(H^{k+2}(\ker\delta))$. Our next step is to show that $F_0'(H^{k+2}(\ker\delta))$ is linearly complementing to $\delta^* (H^{k+1}(T^*M))$ in~$\mathfrak Y$. We will then use this fact together with the diffeomorphism invariance of the cross curvature to complete the proof.

Given $y\in\mathfrak Y$, there exist $h\in H^k(\ker\delta)$ and $v\in H^{k+1}(T^*M)$ such that $y=h+\delta^*v$. The ellipticity of $-\frac{1}{2}\Delta_L-I$ and the absence of $-2$ among the eigenvalues of $\Delta_L$ imply that 
\begin{align*}
h=-\tfrac{1}{2}\Delta_Lu-u
\end{align*}
for some $u\in\mathfrak X$; see~\cite[Theorem~3.12]{Cantor}. By~\eqref{Xsplit}, we may assume that $u$ lies in~$H^{k+2}(\ker\delta)$. Clearly,
\begin{align*}
y=h+\delta v^*&=-\tfrac{1}{2}\Delta_Lu-u+\delta^*v \\
&=-\tfrac{1}{2}\Delta_Lu-\delta^*\delta G(u)-u+\delta^*(v+\delta G(u))=F_0'(u)+\delta^*(v+\delta G(u)).
\end{align*}
We conclude that $F_0'(H^{k+2}(\ker\delta))$ is linearly complementing to $\delta^* (H^{k+1}(T^*M))$ in~$\mathfrak Y$. The latter space is tangent to the orbit of the group of $H^{k+1}$-diffeomorphisms at~$g_0$. Therefore, every vector close to 0 in~$\mathfrak Y$ can be brought into $F_0'(\mathfrak X)$ by the action of this group. More precisely, by~\cite[Lemma~2.3]{DeTurck85}, if $Y$ is close to $g_0$ in~$\mathfrak Y$, there is an $H^{k+1}$-diffeomophism $\phi$ such that $\phi^*(Y)=F(h)$ for some $h\in H^{k+2}(\ker\delta)$. Using the definition of~$F$, we find
\begin{align*}
X\big(\big(\phi^{-1}\big)^*(g_0+h)\big)=Y.
\end{align*}
The inequality $k\ge5$ and the Sobolev embedding theorem imply that $\phi^{-1}\in C^4$ and $h\in C^5$, which means $\big(\phi^{-1}\big)^*(g_0+h)\in C^3$. By Theorem~\ref{regularityxcurvature}, this metric must be smooth.
\end{proof}

\section{Existence for $SO(2)\times SO(2)$-invariant metrics}\label{sec_so2_exist}

The natural action of $SO(2)\times SO(2)$ on $\mathbb R^4=\mathbb{R}^2\times \mathbb{R}^2$ and the standard embedding $\mathbb S^3\hookrightarrow\mathbb R^4$ define a cohomogeneity one action of $SO(2)\times SO(2)$ on~$\mathbb S^3$. The principal orbits of this action are tori $\mathbb{S}^1\times \mathbb{S}^1$; the  two singular orbits are both copies of $\mathbb{S}^1$. 
Suppose $Y$ is a symmetric positive-definite $SO(2)\times SO(2)$-invariant tensor field on~$\mathbb S^3$. One may view $Y$ as a Riemannian metric on~$\mathbb S^3$. We fix a $Y$-geodesic $\gamma:[0,1]\to\mathbb S^3$ connecting the two singular orbits of the $SO(2)\times SO(2)$ action. This gives us a natural diffeomorphism between the principal part of $\mathbb S^3$ and the product $(0,1)\times \mathbb{S}^1\times \mathbb{S}^1$. We choose coordinates $\theta_1,\theta_2\in[0,2\pi)$ on the two $\mathbb S^1$ factors so that $\theta_1=\theta_2=0$ for the points on~$\gamma$. There exist a constant $y_{0}>0$ and smooth functions $y_{1}$, $y_{2}$ and $y_{12}$ on $(0,1)$ such that
\begin{align*}
    Y=y_0^2dr\otimes dr+y_1(r)^2d\theta_1\otimes d\theta_1+y_{2}(r)^2d\theta_2\otimes d\theta_2+y_{12}(r)(d\theta_1\otimes d\theta_2&+d\theta_2\otimes d\theta_1), \\ r&\in(0,1).
\end{align*}
Scaling $Y$ if necessary, we may assume that $y_{0}=1$. We also suppose $Y$ is diagonal, i.e., $y_{12}$ is identically~0. Thus,
\begin{align}\label{Y_form}
    Y=dr\otimes dr+y_{1}(r)^2d\theta_1\otimes d\theta_1+y_{2}(r)^2d\theta_2\otimes d\theta_2,\qquad r\in(0,1).
\end{align}
We are now ready to state the main result of this section.

\begin{theorem}\label{so2existence}
Let $Y$ be a symmetric positive-definite $SO(2)\times SO(2)$-invariant tensor field on~$\mathbb S^3$ of the form~(\ref{Y_form}). Then there exists an $SO(2)\times SO(2)$-invariant Riemannian metric $g$ such that~$X(g)=Y$.
\end{theorem}

The rest of this section is devoted to the proof of Theorem~\ref{so2existence}.

\subsection{The equations of prescribed cross curvature}

Suppose that $h$, $f_1$ and $f_2$ are smooth positive functions on $(0,1)$. Then the expression
\begin{align}\label{so2metric}
    g=h(r)^2dr\otimes dr+f_{1}(r)^2d\theta_1\otimes d\theta_1+f_{2}(r)^2d\theta_2\otimes d\theta_2,\qquad r\in(0,1),
\end{align}
defines an $SO(2)\times SO(2)$-invariant metric on $(0,1)\times\mathbb S^1\times\mathbb S^1$. The following result tells us when this metric can be extended smoothly to~$\mathbb S^3$. While it is well-known, we sketch the proof for the convenience of the reader.

\begin{lemma}\label{lem_so2_smooth}
Equality~(\ref{so2metric}) defines an $SO(2)\times SO(2)$-invariant metric on $\mathbb S^3$ if and only if the following conditions are satisfied:
\begin{enumerate}
    \item
    One can extend $h$ to a smooth function on $(-1,2)$ that is even about both $r=0$ and $r=1$.
    \item
    One can extend $f_1$ to a smooth function on $(-1,2)$ that is odd about $r=0$ and even about~$r=1$.
    \item
    One can extend $f_2$ to a smooth function on $(-1,2)$ that is even about $r=0$ and odd about~$r=1$.
    \item
These extensions satisfy
\begin{align}\label{bc_hf1f2}
f_1'(0)=h(0)>0,\qquad f_1(1)>0,\qquad f_2(0)>0,\qquad f_2'(1)=-h(1)<0.
\end{align}
\end{enumerate}
\end{lemma}

\begin{proof}
We will prove that~\eqref{so2metric} defines a metric near the orbit of $\gamma(0)$ if and only if the conditions at $r=0$ hold. The argument is almost exactly the same near the other singular orbit. Let us introduce local coordinates $(x_1,x_2,x_3)$ centred at $\gamma(0)$ by setting
$$
x_1(r,\theta_1,\theta_2)=r\cos\theta_1,\qquad x_2(r,\theta_1,\theta_2)=r\sin\theta_1,\qquad x_3(r,\theta_1,\theta_2)=\theta_2,
$$
where $r\in[0,1)$ and $\theta_1,\theta_2\in[0,2\pi)$. Denote by $(g_{ij})_{i,j=1}^3$ the components of the metric $g$ given by~\eqref{so2metric} with respect to these coordinates. We compute
\begin{align*}
dr=\cos\theta_1\,dx_1+\sin\theta_1\,dx_2, \qquad d\theta_1=-\frac{\sin\theta_1}{r}dx_1+\frac{\cos\theta_1}{r}dx_2,\qquad d\theta_2=dx_3.
\end{align*}
Consequently,
\begin{align}\label{smoo_comp_met}
g_{11}(x_1(r,\theta_1,\theta_2),x_2(r,\theta_1,\theta_2),x_3(r,\theta_1,\theta_2))&=h^2(r)+\bigg(\frac{f_1^2(r)}{r^2}-h^2(r)\bigg)\sin^2\theta_1, \notag \\
g_{22}(x_1(r,\theta_1,\theta_2),x_2(r,\theta_1,\theta_2),x_3(r,\theta_1,\theta_2))&=h^2(r)+\bigg(\frac{f_1^2(r)}{r^2}-h^2(r)\bigg)\cos^2\theta_1.
\end{align}

Assume that $g$ extends smoothly to the orbit of~$\gamma(0)$. Formula~\eqref{smoo_comp_met} implies that the $C^\infty$ function $\bar h:(-1,1)\to(0,\infty)$ given by
\begin{align*}
\bar h(x)=\sqrt{g_{11}(x,0,0)}
\end{align*}
is even and that $h(r)=\bar h(r)$ for $r\in(0,1)$. Thus, condition~1 is satisfied. Similiar reasoning proves the existence of an even $C^\infty$ function $\bar f_2:(-1,1)\to(0,\infty)$ such that $f_2(r)=\bar f_2(r)$ for $r\in(0,1)$. This means condition~3 holds.

Adding the two equalities in~\eqref{smoo_comp_met}, we obtain
\begin{align*}
(g_{11}+g_{22})(x_1(r,\theta_1,\theta_2),x_2(r,\theta_1,\theta_2),x_3(r,\theta_1,\theta_2))=\frac{f_1^2(r)}{r^2}.
\end{align*}
It follows that $\frac{f_1^2(r)}{r^2}$ extends to an even $C^\infty$ function on $(-1,1)$. The formula
\begin{align*}
\lim_{r\to0}\frac{f_1(r)}{r}=\lim_{x\to0}\sqrt{g_{11}(0,x,0)}=\sqrt{g_{11}(0,0,0)}=\lim_{x\to0}\sqrt{g_{11}(x,0,0)}=\lim_{r\to0}h(r)
\end{align*}
proves that $f_1'(0)=h(0)>0$. Consequently, $f_1$ extends to an odd $C^\infty$ function on $(-1,1)$. Finally,
$$
f_2(0)=\sqrt{g_{33}(0,0,0)}>0,
$$
which proves~\eqref{bc_hf1f2}.

Assume that conditions~1--4 are satisfied at $r=0$. Let $g(\gamma(0))$ be the $(0,2)$-tensor with the matrix
$$
\left(
\begin{matrix}
h^2(0) & 0 & 0 \\
0 & h^2(0) & 0 \\
0 & 0 & f_2^2(0)
\end{matrix}
\right)
$$
in the coordinates~$(x_1,x_2,x_3)$. Together with $SO(2)\times SO(2)$-invariance, this defines an extension of the metric $g$ given by~\eqref{so2metric} to the orbit of~$\gamma(0)$. We must show that this extension is smooth. Conditions~1, 2 and~4 imply that 
$$
h^2(r)=\hat h(r^2)\qquad\mbox{and}\qquad \frac{f_1^2(r)}{r^2}-h^2(r)=r^2\check h(r^2)
$$
for some $C^\infty$ functions $\hat h,\check h:(-1,1)\to\mathbb R$ and all $r\in(0,1)$; see~\cite{HW43}. As a consequence,
$$
g_{11}(x_1,x_2,x_3)=\hat h(x_1^2+x_2^2)+x_2^2 \check h(x_1^2+x_2^2),
$$
which means $g_{11}$ is~$C^\infty$. Analogous arguments prove that the other components of $g$ are $C^\infty$ as well. Thus, $g$ must be smooth near the orbit of~$\gamma(0)$.
\end{proof}

To prove Theorem~\ref{so2existence}, we will search for a solution to the equation $X(g)=Y$ in the class of metrics of the form~\eqref{so2metric}. Our first step is to convert this equation into a system of ODEs. To this end, define
\begin{align*}
    l_1=-\frac{f_1'}{h},&\qquad l_2=-\frac{f_2'}{h},\qquad 
   \phi_1=\frac{y_1}{y_2}, \qquad \phi_2=\frac{y_2}{y_1},\qquad \sigma=y_1y_2.
\end{align*}
Application of Lemma~\ref{lem_so2_smooth} to the tensor field~$Y$ shows that $y_1$ and $y_2$ possess the following properties:
\begin{enumerate}
    \item
    One can extend $y_1$ to a smooth function on $(-1,2)$ that is odd about~$r=0$ and even about~$r=1$.
    \item
    One can extend $y_2$ to a smooth function on $(-1,2)$ that is even about~$r=0$ and odd about~$r=1$.
    \item
These extensions satisfy
\begin{align}\label{BC_y}
    y_1'(0)=1,\qquad y_1(1)>0,\qquad y_2(0)>0,\qquad y_2'(1)=-1.
\end{align}
\end{enumerate}
Consequently, the function $\sigma:(0,1)\to (0,\infty)$ is a restriction of a smooth function on $(-1,2)$ that is odd about both $r=0$ and $r=1$. 
Formulas~\eqref{BC_y} imply
\begin{align}\label{smoothessphi}
&\frac{\phi_1'}{\phi_1}=\frac{1}{r}+S_{\phi_1} , &&\frac{\phi_1}{\sigma}=\frac{1}{y_2(0)^2}+S_{\sigma1}, \notag
\\
&\frac{\phi_2'}{\phi_2}=-\frac{1}{r}+S_{\phi_2}, && \frac{\phi_2}{\sigma}=\frac{1}{r^2}-\frac{y_1'''(0)}{3}+S_{\sigma2}.
\end{align}
Here, $S_{\phi_i}$ and $S_{\sigma i}$ are smooth functions on $[0,1)$ with
\begin{align*}
S_{\phi_i}=O(r),\qquad S_{\sigma i}=O(r^2),\qquad r\to0.
\end{align*}
Equalities similar to~\eqref{smoothessphi} hold near $r=1$. 

\begin{lemma}\label{so2equationslemma}
Suppose that the metric $g$ on $\mathbb S^3$ given by~(\ref{so2metric}) is such that $X(g)=Y$. Then the functions $l_1$ and $l_2$ solve the ODEs
\begin{align}\label{so2equationsli}
    \bigg(\frac{l_1'}{\phi_1}\bigg)'=\frac{l_1l_2^2}{\sigma}, \qquad \bigg(\frac{l_2'}{\phi_2}\bigg)'=\frac{l_1^2l_2}{\sigma},\qquad r\in(0,1),
\end{align}
subject to the boundary conditions
\begin{align}\label{so2smoothnessli}
    l_1(0)=-1, \qquad l_1'(0)=0, \qquad l_2(0)=0,\qquad l_1(1)=0, \qquad
    l_2(1)=1, \qquad l_2'(1)=0.
\end{align}
Conversely, suppose that one can find $l_1$ and $l_2$ that satisfy~(\ref{so2equationsli}) and possess the following properties:
\begin{enumerate}
    \item One can extend $l_1$ to a smooth function on $(-1,2)$ that is even about $r=0$ and odd about $r=1$.
    \item One can extend $l_2$ to a smooth function on $(-1,2)$ that is odd about $r=0$ and even about $r=1$.
    \item These extensions satisfy condition~(\ref{so2smoothnessli}) along with the inequalities $l_2'(0)>0$ and $l_1'(1)>0$.
\end{enumerate}
Then there exists a metric $g$ on $\mathbb S^3$ of the form~(\ref{so2metric}) such that $X(g)=Y$.
\end{lemma}

\begin{proof}
Let $g$ satisfy~\eqref{so2metric}. The Ricci curvature and the scalar curvature of $g$ are given by 
\begin{align*}
    \Ric(g)&=-\left(\frac{f_1''}{f_1}-\frac{h'f_1'}{hf_1}+\frac{f_2''}{f_2}-\frac{h'f_2'}{hf_2}\right)dr\otimes dr\\
    &\hphantom{=}~-\left(\frac{f_1f_1''}{h^2}-\frac{f_1h'f_1'}{h^3}+\frac{f_1f_1'f_2'}{h^2f_2}\right)d\theta_1\otimes d\theta_1-\left(\frac{f_2f_2''}{h^2}-\frac{f_2h'f_2'}{h^3}+\frac{f_2f_2'f_1'}{h^2f_1}\right)d\theta_2\otimes d\theta_2
    \\ &=\left(\frac{hl_1'}{f_1}+\frac{hl_2'}{f_2}\right)dr\otimes dr+\left(\frac{f_1l_1'}{h}-\frac{f_1l_1l_2}{f_2}\right)d\theta_1\otimes d\theta_1+\left(\frac{f_2l_2'}{h}-\frac{f_2l_1l_2}{f_1}\right)d\theta_2\otimes d\theta_2,
    \\
     S(g)&=2\frac{l_1'}{hf_1}+2\frac{l_2'}{hf_2}-2\frac{l_1l_2}{f_1f_2};
\end{align*}
see, e.g.,~\cite[Lemma~3.1]{Pulemotov16}. Consequently,
\begin{align*}
    X(g)=\frac{l_1'l_2'}{f_1f_2}dr\otimes dr-\frac{l_1l_2l_1'}{f_2h}d\theta_1\otimes d\theta_1-\frac{l_1l_2l_2'}{f_1h}d\theta_2\otimes d\theta_2.
\end{align*}
Suppose $X(g)=Y$. Then
\begin{align*}
    \frac{l_1'l_2'}{f_1f_2}=1,\qquad -\frac{l_1l_2l_1'}{f_2h}=y_1^2,\qquad -\frac{l_1l_2l_2'}{f_1h}=y_2^2,
\end{align*}
which implies
\begin{align*}
    \sigma=\pm\frac{l_1l_2}{h}, \qquad
   l_1'=\mp\phi_1f_2,\qquad l_2'=\mp\phi_2f_1.
\end{align*}
By Lemma~\ref{lem_so2_smooth}, $l_1$ and $l_2$ can be extended to smooth functions on $(-1,2)$ satisfying conditions~\eqref{so2smoothnessli}. It is impossible, according to these conditions, to have $l_1 l_2>0$ on all of $(0,1)$. The positivity of $\sigma$ on $(0,1)$ then implies
\begin{align*}
    \sigma=-\frac{l_1l_2}{h}, \qquad
   \bigg(\frac{l_1'}{\phi_1}\bigg)'=f_2'=-l_2h=\frac{l_1l_2^2}{\sigma},\qquad \bigg(\frac{l_2'}{\phi_2}\bigg)'=f_1'=-l_1h=\frac{l_1^2l_2}{\sigma}.
\end{align*}

Conversely, suppose that $l_1$ and $l_2$ solve the boundary-value problem~\eqref{so2equationsli}--\eqref{so2smoothnessli} and possess the properties listed in the lemma. Let us show that $l_1<0$ on $(0,1)$. Assuming the contrary, we pick a point $r_0\in(0,1)$ where $l_1$ attains its maximum. Clearly $l_1(r_0)\ge0$ and $l_1'(r_0)=0$. If $l_1(r_0)$ vanishes, then $l_1$ is identically~0 by uniqueness of solutions to an ODE. However, this contradicts the boundary conditions~\eqref{so2smoothnessli}. Thus, we may assume $l_1(r_0)>0$. Let $\delta$ be the largest number in $(0,1)$ such that $l_1\ge0$ on~$[r_0,\delta)$. The first equation in~\eqref{so2equationsli} implies
\begin{align*}
l_1''=\frac{l_1'\phi_1'}{\phi_1}+
    \frac{l_1l_2^2\phi_1}{\sigma}.
\end{align*}
Consequently, $l_1'\ge0$ on $[r_0,\delta)$, which means $l_1(\delta)\ge l_1(r_0)>0$. However, this is impossible in light of the maximality of $\delta$ and the boundary conditions~\eqref{so2smoothnessli}. We conclude that $l_1<0$ on $(0,1)$. Analogous reasoning shows that $l_2>0$ on $(0,1)$. 

Define
\begin{align*}
h=-\frac{l_1l_2}{\sigma},\qquad r\in(0,1).
\end{align*}
Evidently, this function is positive on $(0,1)$. The boundary conditions~\eqref{so2smoothnessli} imply that
\begin{align*}
h(r)=\frac{l_2'(0)r+o(r)}{y_2(0)r+o(r)},\qquad r\to0.
\end{align*}
Set $h(0)=\lim_{r\to0}h(r)>0$. Properties of $l_1$, $l_2$ and $\sigma$ imply that $h$ can be extended to a function that is even about $r=0$. Analogously, set $h(1)=\lim_{r\to1}h(r)>0$. It is clear that $h$ can be extended to a function that is even about $r=1$.

Finally, define
\begin{align*}
f_1(r)=-\int_0^rl_1(t)h(t)\,dt,\qquad f_2(r)=\int_r^1l_2(t)h(t)\,dt,\qquad r\in[0,1].
\end{align*}
These functions are positive on $(0,1)$. The boundary conditions~\eqref{so2smoothnessli} imply formulas~\eqref{bc_hf1f2}. It is easy to see that~\eqref{so2metric} yields a metric $g$ on $\mathbb S^3$ satisfying $X(g)=Y$.
\end{proof}

\subsection{The linear problem}
In this subsection, we study the existence, uniqueness and regularity properties of the ODEs that arise as the principal part of the linearised prescribed cross curvature equations. We will need the following two simple lemmas from classical analysis. We include the proofs for the convenience of the reader.

\begin{lemma}\label{DDT}
Let $u$ be a $C^k$-differentiable function on $[0,T)$ for some $k\in\mathbb N$ and $T>0$. Assume that $u(0)=0$. The limit $\lim_{t\to0}\frac{u(t)}{t}$ exists, and the continuous function $v:[0,T)\to\mathbb R$ such that
\begin{align*}
v(t)=\frac{u(t)}{t},\qquad t\in(0,T),
\end{align*}
is $C^{k-1}$-differentiable on $[0,T)$. 
\end{lemma}

\begin{proof}
By adding a polynomial to~$u$, we may assume without loss of generality that  $u^{(i)}(0)$ vanishes for $i=0,1,\ldots,k$. The existence of $\lim_{t\to0}\frac{u(t)}{t}$ is a consequence of l'H\^opital's rule. We will prove that
\begin{align}\label{simple_ders}
v^{(j)}(0)=\lim_{t\to0}v^{(j)}(t)=0
\end{align}
for $j=0,1,\ldots,k-1$. The assertion of the lemma will follow immediately. We proceed by induction. Since $v$ is continuous,
\begin{align*}
v(0)=\lim_{t\to0}v(t)=\lim_{t\to0}\frac{u(t)}t=u'(0)=0.
\end{align*}
Assume that formula~\eqref{simple_ders} holds for $j=m\le k-2$. Using l'H\^opital's rule repeatedly, we obtain
\begin{align*}
v^{(m+1)}(0)&=\lim_{t\to0}\frac{v^{(m)}(t)}{t}=\lim_{t\to0}\sum_{i=0}^m(-1)^{m-i}(m-i)!{m\choose i}\frac{u^{(i)}(t)}{t^{m-i+2}}=0, \\
\lim_{t\to0}v^{(m+1)}(t)&=\lim_{t\to0}\sum_{i=0}^{m+1}(-1)^{m+1-i}(m+1-i)!{m+1\choose i}\frac{u^{(i)}(t)}{t^{m-i+2}}=0.
\end{align*}
Thus,~\eqref{simple_ders} holds for $j=0,1,\ldots,k-1$.
\end{proof}

\begin{lemma}\label{IIOR}
Let $u$ be a $C^k$-differentiable function on $[0,T)$ for some $k\in \mathbb{N}$ and $T>0$. Given $i\in \mathbb{N}$, the limit 
$$
\lim_{t\to 0}\frac{1}{t^i}\int_0^t \tau^i u(\tau)\,d\tau
$$
exists. The continuous function $v:[0,T)\to \mathbb{R}$ such that 
\begin{align*}
v(t)=\frac{1}{t^i}\int_0^t \tau^i u(\tau)\,d\tau, \qquad t\in (0,T),
\end{align*}
is $C^{k+1}$-differentiable on $[0,T)$. 
\end{lemma}

\begin{proof}
We proceed by induction. If $i=0$, the result is obvious. Suppose that it holds for all $i\le j$ with $j\in \mathbb{N}\cup\{0\}$. Given a $C^k$-differentiable $u$ on~${[0,T)}$, define $U(t)=\int_0^t u(\tau)\,d\tau$. Integration by parts yields 
\begin{align*}
\frac{1}{t^{j+1}}\int_0^{t}\tau^{j+1}u(\tau)\,d\tau=U(t)-\frac{j+1}{t^{j+1}}\int_0^t\tau^{j}U(\tau)\,d\tau, \qquad t\in(0,T).
\end{align*}
Clearly, $U(t)$ is $C^{k+1}$-differentiable on $[0,T)$. By the induction hypothesis, the map
$$
t\mapsto\frac{1}{t^j}\int_0^t \tau^j U(\tau)\,d\tau
$$
extends to a $C^{k+2}$ function $V:[0,T)\to\mathbb R$. Moreover,
\begin{align*}
|V(0)|=\lim_{t\to 0}\frac1{t^j}\bigg|\int_0^t\tau^jU(\tau)d\tau\bigg|\le \lim_{t\to 0} \int_0^t|U(\tau)|\,d\tau=0.
\end{align*}
Applying Lemma~\ref{DDT}, we easily conclude that
\begin{align*}
t\mapsto\frac{1}{t^{j+1}}\int_0^{t}\tau^{j+1}u(\tau)\,d\tau=U(t)-\frac{(j+1)V(t)}{t}.
\end{align*}
extends to a $C^{k+1}$ function on $[0,T)$.
\end{proof}

To prove Theorem~\ref{so2existence}, it suffices to find $l_1$ and $l_2$ satisfying~\eqref{so2equationsli}--\eqref{so2smoothnessli} along with conditions~1--3 of Lemma~\ref{so2equationslemma}. Our first step in this direction is to establish a short-time existence and uniqueness result for two relatively simple linear ODEs with singularities near the initial time. These ODEs are the principal part of the linearisation of system~\eqref{so2equationsli}.

\begin{lemma}\label{IV_linearised}
Let $S_1$ and $S_2$ be functions on $[0,T]$ for some $T>0$. Given $k\in[0,\infty)\cap\mathbb Z$, assume that $S_1$ and $tS_2$ are $C^k$-differentiable on $[0,T]$. There exist unique $C^{k+2}$ functions $s_1$ and $s_2$ on $[0,T]$ such that
\begin{align}\label{lin_IVP}
    &s_1''-\frac{s_1'}{t}=tS_1, && s_1''(0)=s_1'(0)=s_1(0)=0, \notag \\
    &s_2''+\frac{s_2'}{t}-\frac{s_2}{t^2}=tS_2, && s_2'(0)=s_2(0)=0.
\end{align}
Moreover,
\begin{align}\label{apriori_est_so2}
\sup_{t\in(0,T]}\bigg(\frac{|s_i(t)|}{t^2}+\frac{\left|s_i'(t)\right|}{t}\bigg)&\le CT\sup_{t\in(0,T]}|S_i(t)|,\qquad i=1,2,
\end{align}
where $C>0$ is independent of $S_1$, $S_2$ and $T$.
\end{lemma}

\begin{proof}
We can write the ODEs for $s_1$ and $s_2$ as
\begin{align*}
    \Big(\frac{s_1'}{t}\Big)'=S_1, \qquad \bigg(\frac{(ts_2)'}{t}\bigg)'=tS_2.
\end{align*}
Solving explicitly and using the initial conditions in~\eqref{lin_IVP}, we obtain
\begin{align*}
s_1(t)=\int_0^t\int_0^\tau\tau S_1(\rho)\,d\rho\,d\tau,\qquad
s_2(t)=\frac1t\int_0^t\int_0^\tau\tau\rho S_2(\rho)\,d\rho\,d\tau,\qquad t\in(0,T].
\end{align*}
Clearly, estimate~\eqref{apriori_est_so2} holds. It is obvious that $s_1$ is $C^{k+2}$. The fact that $s_2$ is $C^{k+2}$ follows from Lemma \ref{IIOR} applied to the $C^{k+1}$ function $ \int_0^{\tau}\rho S_2(\rho)\,d\rho$.
\end{proof}

\subsection{Existence near the singular orbits}

For $p\in[0,1]$, consider the system of ODEs
\begin{align}\label{MECFSO2}
    \left(\frac{l_1'}{\phi_1}\right)'&=\frac{l_1(pl_2+(1-p)\sin\frac{\pi t}{2})^2}{\sigma}, \notag \\
    \left(\frac{l_2'}{\phi_2}\right)'&=\frac{l_2(pl_1-(1-p)\cos\frac{\pi t}{2})^2}{\sigma}.
\end{align}
This system reduces to~\eqref{so2equationsli} if~$p=1$. We will prove its solvability near the endpoints of the interval~$[0,1]$. To do so, we will linearise it and employ Lemma~\ref{IV_linearised} along with the Banach fixed point theorem. Afterwards, we will vary $p$ from 0 to~1 and use Brouwer degree theory to produce $l_1$ and $l_2$ that solve~\eqref{so2equationsli}--\eqref{so2smoothnessli}. From these $l_1$ and $l_2$, we will construct a metric of the form~\eqref{so2metric} satisfying the equation~$X(g)=Y$.

\begin{lemma}\label{IVPWD}
Given $\alpha_1,\alpha_2\in \mathbb{R}$, system~(\ref{MECFSO2}) has a unique solution on $(0,T]$ for some $T\in\big(0,\frac12\big]$ with 
\begin{align}\label{alphaICs}
l_1(0)=-1,\qquad l_1'(0)=l_2(0)=0,\qquad l_1''(0)=\alpha_1,\qquad  l_2'(0)=\alpha_2.
\end{align}
The values of $l_1$, $l_1'$, $l_2$ and $l_2'$ at every $t\in(0,T]$ depend continuously on $\alpha_1,\alpha_2\in \mathbb{R}$ and $p\in [0,1]$. Similarly, given $\beta_1,\beta_2\in \mathbb{R}$, system~(\ref{MECFSO2}) has a unique solution on $[1-T,1)$ for some $T\in\big(0,\frac12]$ with 
\begin{align}\label{betaICs}
l_1(1)=l_2'(1)=0,\qquad l_2(1)=1, \qquad l_1'(1)=\beta_1,\qquad  l_2''(1)=-\beta_2.
\end{align}
Again, the values of $l_1$, $l_1'$, $l_2$ and $l_2'$ depend continuously on $\beta_1,\beta_2\in\mathbb R$ and $p\in [0,1]$. 
\end{lemma}

\begin{proof}
Let us prove the existence of $l_1$ and $l_2$ satisfying~\eqref{MECFSO2} near $t=0$ with the initial conditions specified in the lemma. Consider the system
\begin{align}\label{MECFSO2'}
    s_1''-\frac{s_1'\phi_1'}{\phi_1}&=\frac{(2s_1+\alpha_1t^2-2)\phi_1\mathcal C_1(t,s_2,\alpha_2,p)^2}{2\sigma}+\frac{\alpha_1t\phi_1'}{\phi_1}-\alpha_1, \notag \\
    s_2''-\frac{s_2'\phi_2'}{\phi_2}&=\frac{(s_2+\alpha_2t)\phi_2\mathcal C_2(t,s_1,\alpha_1,p)^2}{4\sigma}+\frac{\alpha_2\phi_2'}{\phi_2},
\end{align}
where
\begin{align*}
\mathcal C_1(t,s_2,\alpha_2,p)&=ps_2+p\alpha_2t+(1-p)\sin\tfrac{\pi t}{2}, \\
\mathcal C_2(t,s_1,\alpha_1,p)&=2ps_1+p\alpha_1t^2-2p-2(1-p)\cos\tfrac{\pi t}{2}.
\end{align*}
One can recover solutions to~\eqref{MECFSO2} from those to~\eqref{MECFSO2'} by setting
\begin{align*}
l_1=s_1+\frac{\alpha_1t^2}2-1,\qquad l_2=s_2+\alpha_2t. 
\end{align*}
Thus, it suffices to prove the existence of $s_1$ and $s_2$ satisfying~\eqref{MECFSO2'} with the initial conditions
\begin{align}\label{ini_cond_homog}
    s_1''(0)=s_1'(0)=s_1(0)=0,\qquad s_2'(0)=s_2(0)=0.
\end{align}
Formulas~\eqref{smoothessphi} enable us to transform~\eqref{MECFSO2'} as
\begin{align*}
s_1''-\frac{s_1'}{t}&=\frac12(2s_1+\alpha_1t^2-2)\Big(\frac1{y_2(0)^2}+S_{\sigma1}\Big)\mathcal C_1(t,s_2,\alpha_2,p)^2+(s_1'+\alpha_1t)S_{\phi_1},
\\
s_2''+\frac{s_2'}{t}-\frac{s_2}{t^2}&=\frac14\Big(\frac{s_2}{t^2}+\frac{\alpha_2}t\Big)(\mathcal C_2(t,s_1,\alpha_1,p)-2)(\mathcal C_2(t,s_1,\alpha_1,p)+2)
\\
&\hphantom{=}~+
\frac14(s_2+\alpha_2t)\Big(S_{\sigma2}-\frac{y_1'''(0)}3\Big)\mathcal C_2(t,s_1,\alpha_1,p)^2+(s_2'+\alpha_2)S_{\phi_2}.
\end{align*}
Straightforward verification shows that
\begin{align*}
\mathcal C_2(t,s_1,\alpha_1,p)+2=2ps_1+O(t^2),\qquad t\to0.
\end{align*}
This means we can find smooth functions $\mathcal S_1$ and $\mathcal S_2$ from $\mathbb R^7$ to $\mathbb R$ such that~\eqref{MECFSO2'} is equivalent to
\begin{align}\label{sys_FixedPT}
    s_1''-\frac{s_1'}{t}&=t \mathcal S_1\big(t,s_1',s_1,\tfrac{s_2}t,\alpha_1,\alpha_2,p\big), \notag \\
    s_2''+\frac{s_2'}{t}-\frac{s_2}{t^2}&=t\mathcal S_2\big(t,s_2',\tfrac{s_1}{t^2},\tfrac{s_2}{t},\alpha_1,\alpha_2,p\big).
\end{align}
We will prove the solvability of this system subject to the initial conditions~\eqref{ini_cond_homog}. The existence of $l_1$ and $l_2$ satisfying~\eqref{MECFSO2}--\eqref{alphaICs} will follow immediately.

Choose $T\in(0,1)$. Let $\mathfrak B$ stand for the completion of the space of pairs $(u,v)$ of smooth compactly supported functions on $(0,T]$ with respect to the norm
\begin{align*}
     \|(u,v)\|_{\mathfrak B}=\sup_{t\in (0,T]}\left(\frac{|u(t)|}{t^2}+\frac{|u'(t)|}{t}+\frac{|v(t)|}{t^2}+\frac{|v'(t)|}{t}\right).
\end{align*}
Denote by $\mathfrak B_1$ the unit ball in $\mathfrak B$ centred at~0. We will apply the Banach fixed point theorem in $\mathfrak B_1$ to prove the solvability of~\eqref{sys_FixedPT}. Let $L$ be the map that takes $(u,v)\in\mathfrak B$ to the unique pair $(x,y)$ of $C^2$ functions on $[0,T]$ satisfying
\begin{align*}
&x''-\frac{x'}{t}=t \mathcal S_1\big(t,u',u,\tfrac{v}t,\alpha_1,\alpha_2,p\big), && x''(0)=x'(0)=x(0)=0,\\
&y''+\frac{y'}{t}-\frac{y}{t^2}=t\mathcal S_2\big(t,v',\tfrac{u}{t^2},\tfrac{v}{t},\alpha_1,\alpha_2,p\big), && y'(0)=y(0)=0.
\end{align*}
Lemma~\ref{IV_linearised} guarantees the existence of such a pair. Estimate~\eqref{apriori_est_so2} implies that $L$ is a contraction on~$\mathfrak B_1$ if $T$ is sufficiently small. Thus, $L$ has a fixed point satisfying~\eqref{sys_FixedPT}. We conclude that system~\eqref{MECFSO2} has a solution $(l_1,l_2)$ near $t=0$ subject to the initial conditions specified in the lemma. It remains to prove uniqueness and continuous dependence on $\alpha_1$, $\alpha_2$ and~$p$.

Consider a pair $(s_1,s_2)$ satisfying~\eqref{sys_FixedPT} and~\eqref{ini_cond_homog}. Given $\epsilon>0$, choose $\alpha_1^\epsilon,\alpha_2^\epsilon\in\mathbb R$ and $p^\epsilon\in[0,1]$ such that
\begin{align*}
|\alpha_1-\alpha_1^\epsilon|+|\alpha_2-\alpha_2^\epsilon|+|p-p^\epsilon|\le\epsilon. 
\end{align*}
Fix a number $\alpha^{*}$ greater than
$\max\{\left|\alpha_1\right|,\left|\alpha_2\right|,\left|\alpha_1^\epsilon\right|,\left|\alpha_2^\epsilon\right|\}$. Consider a pair $(s_1^\epsilon,s_2^\epsilon)$ such that~\eqref{sys_FixedPT} and~\eqref{ini_cond_homog} hold with $s_1$, $s_2$, $\alpha_1$, $\alpha_2$ and $p$ replaced by~$s_1^\epsilon$, $s_2^\epsilon$, $\alpha_1^\epsilon$, $\alpha_2^\epsilon$ and $p^\epsilon$. Estimate~\eqref{apriori_est_so2} and the smoothness of the functions $\mathcal S_i$ on $\mathbb R^7$ imply that
\begin{align*}
\sum_{i=1}^2\sup_{t\in(0,T^*]}&\bigg(\frac{|s_i(t)-s_i^\epsilon(t)|}{t^2}+\frac{|(s_i(t)-s_i^\epsilon(t))'|}{t}\bigg)
\\ &\le T^*C^*\sum_{i=1}^2\sup_{t\in(0,T^*]}\bigg(\frac{|s_i(t)-s_i^\epsilon(t)|}{t^2}+\frac{|(s_i(t)-s_i^\epsilon(t))'|}{t}\bigg) \\ 
&\hphantom{=}~+T^*C^*\sum_{i=1}^2|\alpha_i-\alpha_i^\epsilon|+T^*C^*|p-p^\epsilon|
\end{align*}
for every $T^*\le T$ and some $C^*>0$ depending only on $\alpha^*$ and~$\mathcal S_i$. If $T^*$ is sufficiently small, then
\begin{align*}
\sum_{i=1}^2\sup_{t\in(0,T^*]}\bigg(\frac{|s_i(t)-s_i^\epsilon(t)|}{t^2}+\frac{|(s_i(t)-s_i^\epsilon(t))'|}{t}\bigg)\le \sum_{i=1}^2|\alpha_i-\alpha_i^\epsilon|+|p-p^\epsilon|.
\end{align*}
Consequently, system~\eqref{sys_FixedPT} has at most one solution on $[0,T^*]$ subject to conditions~\eqref{ini_cond_homog}. This solution depends continuously on $\alpha_1$, $\alpha_2$ and~$p$. Using standard ODE theory, we can easily conclude that uniqueness holds on $[0,T]$. We also obtain continuous dependence on $\alpha_1$, $\alpha_2$ and~$p$ on this interval.
The same arguments can be used to produce solutions on $[1-T,1)$ for some $T>0$ with the specified conditions at $t=1$. 
\end{proof}

\subsection{A priori estimates for the global problem}

To extend the functions $l_1$ and $l_2$ constructed in Lemma~\ref{IVPWD} to larger intervals, we consider a slight modification of system~\eqref{MECFSO2}. Namely, suppose that $F$ is a smooth nonnegative function such that $F(x)=x^2$ if $|x|\le2$, $F(x)=0$ if $|x|\ge3$, and $F(x)\le 8$ for all $x\in\mathbb R$. The following result shows that system~\eqref{MECFSO2} is equivalent to
\begin{align}\label{so2eqs_F}
\bigg(\frac{l_1'}{\phi_1}\bigg)'=\frac{l_1F(pl_2+(1-p)\sin\frac{\pi t}{2})}{\sigma}, \qquad \bigg(\frac{l_2'}{\phi_2}\bigg)'=\frac{F(pl_1-(1-p)\cos\frac{\pi t}{2})l_2}{\sigma},
\end{align}
when conditions~\eqref{so2smoothnessli} are imposed at $t=0$ and~$t=1$. In a sense, it is the maximum principle for~\eqref{so2eqs_F}.

\begin{lemma}\label{so2_lem_max_princ}
Assume that $l_1$ and $l_2$ satisfy~(\ref{so2eqs_F}) on $(0,1)$ with the boundary conditions 
\begin{align}\label{half_BC0}
l_1(0)=-1,\qquad l_2(0)=0,
\qquad l_1(1)=0,\qquad l_2(1)=1. 
\end{align}
Then $-1<l_1<0$ and $0<l_2<1$ on $(0,1)$.
\end{lemma}

The lemma implies that solutions to~\eqref{so2eqs_F}--\eqref{half_BC0} also solve~(\ref{MECFSO2}).

\begin{proof}
We argue as in the proof of Lemma \ref{so2equationslemma}.
Let us show that $l_1<0$ on $(0,1)$. Assuming the contrary, we pick a point $r_0\in(0,1)$ where $l_1$ attains its maximum. If $l_1(r_0)$ vanishes, then $l_1$ is identically~0 by uniqueness of solutions to an ODE. However, this contradicts~\eqref{half_BC0}. Thus, we may assume $l_1(r_0)>0$. Let $\delta$ be the largest number in $[0,1-r_0]$ such that $l_1\ge0$ on~$[r_0,r_0+\delta)$. The first equation in~\eqref{so2eqs_F} implies
\begin{align*}
l_1''=\frac{l_1'\phi_1'}{\phi_1}+
    \frac{l_1F(pl_2+(1-p)\sin\frac{\pi t}{2})\phi_1}{\sigma}.
\end{align*}
Consequently, $l_1'\ge0$ on $[r_0,r_0+\delta)$, which means $l_1(r_0+\delta)\ge l_1(r_0)>0$. However, this is impossible in light of the maximality of $\delta$ and the conditions~\eqref{half_BC0}. We conclude that $l_1<0$ on $(0,1)$. Analogous reasoning shows that $l_1>-1$ and $0<l_2<1$.
\end{proof}

Lemma~\ref{IVPWD} and the global Lipschitz continuity of $F$ on $\mathbb R$ imply that system~\eqref{so2eqs_F} possesses a unique solution on $(0,\frac12]$ such that conditions~\eqref{alphaICs} hold and a unique solution on $[\frac12,1)$ such that~\eqref{betaICs} hold. We denote these solutions $(l^{p,\alpha}_{1-},l^{p,\alpha}_{2-})$ and  $(l^{p,\beta}_{1+},l^{p,\beta}_{2+})$, respectively. They depend continuously on $\alpha=(\alpha_1,\alpha_2)$,  
$\beta=(\beta_1,\beta_2)$ and $p$. We will show that, for certain values of $\alpha_1$, $\alpha_2$, $\beta_1$ and $\beta_2$, they can be combined to produce a solution to problem~\eqref{so2equationsli}--\eqref{so2smoothnessli} on~$[0,1]$. In light of Lemma~\ref{so2equationslemma}, this will complete the proof of Theorem~\ref{so2existence}.

Consider the function $G_p:\mathbb{R}^4\to \mathbb{R}^4$ given by 
\begin{align*}
G_p(\alpha,\beta)&=G_p(\alpha_1,\alpha_2,\beta_1,\beta_2)
\\
&=\big((l^{p,\beta}_{1+}-l^{p,\alpha}_{1-})(\tfrac12),(l^{p,\beta}_{1+}-l^{p,\alpha}_{1-})'(\tfrac12),(l^{p,\beta}_{2+}-l^{p,\alpha}_{2-})(\tfrac12),(l^{p,\beta}_{2+}-l^{p,\alpha}_{2-})'(\tfrac12)\big).
\end{align*}
We will show that $G_1$ has a zero at some $(\hat\alpha,\hat\beta)=(\hat\alpha_{1},\hat\alpha_{2},\hat\beta_{1},\hat\beta_{2})\in\mathbb R^4$. Then the following smooth functions would solve~\eqref{so2eqs_F} on $(0,1)$, subject to conditions~\eqref{so2smoothnessli}:
\begin{align}\label{sol_formula}
l_1=
\begin{cases}
l^{1,\hat\alpha}_{1-} & \mbox{on}~\big[0,\frac12\big], \\ l^{1,\hat\beta}_{1+} & \mbox{on}~\big[\frac12,1\big],
\end{cases}
\qquad
l_2=
\begin{cases}
l^{1,\hat\alpha}_{2-} & \mbox{on}~\big[0,\frac12\big], \\ l^{1,\hat\beta}_{2+} & \mbox{on}~\big[\frac12,1\big].
\end{cases}
\end{align}
By Lemma~\ref{so2_lem_max_princ}, they would satisfy~\eqref{so2equationsli}--\eqref{so2smoothnessli}.

We will use Brouwer degree theory to prove the existence of $(\hat\alpha,\hat\beta)$. This requires showing that the zeros of $G_p$ are all contained in some bounded set. Denote by $\Omega_R$ the open disc in $\mathbb R^2$ of radius $R>0$ centred at the origin.

\begin{lemma}\label{so2uniformbounds}
There exists a number $R_0>0$, independent of $p\in [0,1]$, such that every zero of the function $G_p$ lies in $\Omega_{R_0}\times\Omega_{R_0}$.
\end{lemma}

\begin{proof}
\noindent\textit{Step 1.}
Assume that the assertion of the lemma fails to hold. Then there exist sequences $(\alpha_{k1})_{k=1}^\infty$, $(\alpha_{k2})_{k=1}^\infty$, $(\beta_{k1})_{k=1}^\infty$, $(\beta_{k2})_{k=1}^\infty$ and $(p_k)_{k=1}^\infty\subseteq [0,1]$ such that
\begin{align*}
G_{p_k}(\alpha_{k1},\alpha_{k2},\beta_{k1},\beta_{k2})=0,\qquad \lim_{k\to\infty}\max\{\alpha_{k1},\alpha_{k2},\beta_{k1},\beta_{k2}\}=\infty.
\end{align*}
Let $l_{k1}$ and $l_{k2}$ be functions on $[0,1]$ that solve equations~\eqref{so2eqs_F}, subject to the conditions~\eqref{so2smoothnessli}, and satisfy the formulas
\begin{align*}
l_{k1}''(0)=\alpha_{k1},\qquad l_{k2}'(0)=\alpha_{k2},\qquad l_{k1}'(1)=\beta_{k1},\qquad l_{k2}''(1)=-\beta_{k1}.
\end{align*}
Lemma~\ref{so2_lem_max_princ} implies that $|l_{k1}|\le1$ and $|l_{k2}|\le1$ for all $k\in\mathbb N$. We will show that, after appropriate re-scaling, $l_{k1}$ and $l_{k2}$ or some of their derivatives converge uniformly on every bounded interval. This will lead to a contradiction, completing the proof.

\vspace{5pt}
\noindent\textit{Step 2.}
Let us introduce some notation. Assume that
\begin{align}\label{lim_alpha12}
\lim_{k\to\infty}\max\{\alpha_{k1},\alpha_{k2}\}=\infty.
\end{align}
The case where $\max\{\alpha_{k1},\alpha_{k2}\}$ remains bounded but $\max\{\beta_{k1},\beta_{k2}\}$ goes to infinity is analogous. Choose $\tau_k\in\big[0,\frac12\big]$ so that 
\begin{align*}
|l_{k1}'(\tau_k)|^2+|l_{k1}''(\tau_k)|+|l_{k2}'(\tau_k)|^2=\max_{\tau\in[0,\frac12]}\big(|l_{k1}'(\tau)|^2+|l_{k1}''(\tau)|+|l_{k2}'(\tau)|^2\big).
\end{align*}
Uniqueness of solutions to ordinary differential equations and conditions~\eqref{so2smoothnessli} imply that the quantity on the right-hand side is strictly positive. Denote
\begin{align*}
\lambda_k=\frac1{\sqrt{|l_{k1}'(\tau_k)|^2+|l_{k1}''(\tau_k)|+|l_{k2}'(\tau_k)|^2}}.
\end{align*}
By passing to subsequences if necessary, we may assume that $(\lambda_k)_{k=1}^\infty$, $(\alpha_{k1}\lambda_k^2)_{k=1}^\infty$, $(\alpha_{k2}\lambda_k)_{k=1}^\infty$ and $\big(\frac{\tau_k}{\lambda_k}\big)_{k=1}^\infty$ are all monotone. Formula~\eqref{lim_alpha12} implies that $(\lambda_k)_{k=1}^\infty$ goes to~0. Moreover, $(\alpha_{k1}\lambda_k^2)_{k=1}^\infty$ and $(\alpha_{k2}\lambda_k)_{k=1}^\infty$ are convergent in $[0,1]$. 

\vspace{5pt}\noindent\textit{Step 3.} Assume that
\begin{align}\label{case1_so2_est}
\lim_{k\to\infty}\frac{\tau_k}{\lambda_k}<\infty.
\end{align}
Consider the functions $\tilde{l}_{k1}$ and $\tilde{l}_{k2}$ obtained from $l_{k1}$ and $l_{k2}$ by re-scaling the independent variable as follows:
\begin{align*}
\tilde{l}_{k1}(t)=l_{k1}(\lambda_k t), \qquad \tilde{l}_{k2}(t)=l_{k2}(\lambda_k t),\qquad t\in\big[0,\tfrac1{\lambda_k}\big].
\end{align*}
Our next goal is to show that $\tilde l_{k1}$ converges to $-1$ in the $C^2$-norm on $[0,T]$ for every $T>0$. Clearly,
\begin{align}\label{norm_le1}
\max_{t\in[0,\frac1{2\lambda_k}]}\big(|\tilde l_{k1}'(t)|^2+|\tilde l_{k1}''(t)|+|\tilde l_{k2}'(t)|^2\big) =\lambda_k^2\max_{\tau\in[0,\frac12]}\big(|l_{k1}'(\tau)|^2+|l_{k1}''(\tau)|+|l_{k2}'(\tau)|^2\big)=1.
\end{align}
In light of~\eqref{so2smoothnessli},
\begin{align}\label{l1l2_MVT}
|\tilde l_{k1}'(t)|\le t\max_{\tau\in[0,t]}|\tilde l_{k1}''(\tau)|\le t,\qquad |\tilde l_{k2}(t)|\le t\max_{\tau\in[0,t]}|\tilde l_{k2}'(\tau)|\le t, \qquad t\in\big[0,\tfrac1{2\lambda_k}\big].
\end{align}
The function
\begin{align*}
s_{k1}(t)=\tilde l_{k1}(t)-\frac{\alpha_{k1}\lambda_k^2t^2}2+1
\end{align*}
satisfies the formula
\begin{align*}
\bigg|s_{k1}''(t)-\frac{s_{k1}'(t)}{t}\bigg|&=\bigg|\tilde{l}_{k1}''(t)-\frac{\tilde{l}_{k1}'(t)}{t}\bigg|=\lambda_k^2\bigg|l_{k1}''(\tau)-\frac{l_{k1}'(\tau)}{\tau}\bigg|\bigg|_{\tau=\lambda_kt}
\\
&=\lambda_k^2 \bigg|\frac{S_{\phi_1}(\lambda_k t)\tilde{l}_{k1}'(t)}{\lambda_k}+\frac{\phi_1(\lambda_k t)}{\sigma(\lambda_k t)}\tilde{l}_{k1}(t)\big(p\tilde{l}_{k2}(t)+(1-p)\sin\tfrac{\lambda_k\pi t}{2}\big)^2\bigg|
\\
&\le\lambda_k|S_{\phi_1}(\lambda_k t)|t+\lambda_k^2\bigg|\frac{\phi_1(\lambda_k t)}{\sigma(\lambda_k t)}\bigg|\left(pt+(1-p)\big|\sin\tfrac{\lambda_k\pi t}{2}\big|\right)^2 \le C_1\lambda_k t,
\qquad t\in\big(0,\tfrac1{2\lambda_k}\big],
\end{align*}
where $S_{\phi_1}$ is given by~\eqref{smoothessphi} and $C_1>0$ is a constant independent of~$k$. Lemma~\ref{IV_linearised} implies that
\begin{align*}
\lim_{k\to\infty}\sup_{t\in (0,T]}\bigg(\frac{|s_{k1}(t)|}{t^2}+\frac{\left|s_{k1}'(t)\right|}{t}\bigg)=0
\end{align*}
for each $T>0$. This means
$(\tilde l_{k1})_{k=1}^\infty$ converges to $\frac{\alpha_1t^2}2-1$ in the $C^2$-norm on $[0,T]$ as $k\to\infty$ with
\begin{align*}
\alpha_1=\lim_{k\to\infty}\alpha_{k1}\lambda_k^2\in [0,1].
\end{align*}
Moreover, if $\alpha_1\ne0$, then
\begin{align*}
\lim_{k\to\infty}\tilde l_{k1}\big(\tfrac3{\sqrt{\alpha_1}}\big)=\tfrac72>1,
\end{align*}
which contradicts Lemma~\ref{so2_lem_max_princ}. Thus, $(\tilde l_{k1})_{k=1}^\infty$ converges to $-1$ in the $C^2$-norm on $[0,T]$.

\vspace{5pt}\noindent\textit{Step 4.}
Assume, as above, that~\eqref{case1_so2_est} holds. Our next goal is to show that $\tilde l_{k2}$ converges to $0$. Keeping~\eqref{norm_le1}, \eqref{l1l2_MVT} and~\eqref{so2smoothnessli} in mind, we observe that the function
\begin{align*}
s_{k2}(t)=\tilde l_{k2}(t)-\alpha_{k2}\lambda_kt
\end{align*}
satisfies
\begin{align*}
\bigg|s_{k2}''(t)+\frac{s_{k2}'(t)}{t}-\frac{s_{k2}(t)}{t^2}\bigg|&=\lambda_k^2\bigg|l_2''(\tau) +\frac{l_2'(\tau)}{\tau}-\frac{l_2(\tau)}{\tau^2}\bigg|\bigg|_{\tau=\lambda_k t} 
\\
&\le
\lambda_k^2\bigg|\frac{S_{\phi_2}(\lambda_kt)\tilde{l}'_{k2}(t)}{\lambda_k}+\frac{\tilde{l}_{k2}(t)}{\lambda_k^2 t^2}\big(\big(p\tilde{l}_{k1}(t)-(1-p)\cos\tfrac{\lambda_k\pi t}{2}\big)^2-1\big)\bigg|\\
&\hphantom{=}~+\lambda_k^2\bigg|-\frac{y_1'''(0)}3+S_{\sigma_2}(\lambda_kt)\bigg|\big|\tilde{l}_{k2}(t)\big(p\tilde{l}_{k1}(t)-(1-p)\cos\tfrac{\lambda_k\pi t}{2}\big)^2\big|
\\
&\le
\lambda_k^2t\sup_{\tau\in[0,\frac12]}\bigg|\frac{S_{\phi_2}(\tau)}{\tau}\bigg|+\lambda_k^2t\sup_{\tau\in[0,\frac12]}\bigg|-\frac{y_1'''(0)}3+S_{\sigma_2}(\tau)\bigg|
\\
&\hphantom{=}~+\frac2{t}\big|p(\tilde{l}_{k1}(t)+1)-(1-p)\big(\cos\tfrac{\lambda_k\pi t}{2}-1\big)\big)\big|
\\
&\le C_2 \lambda_k^2t+\frac2{t}|\tilde{l}_{k1}(t)+1|
\\
&\le C_2\lambda_k^2t+2t\max_{\tau\in[0,t]}|\tilde l_{k1}''(\tau)|,\qquad t\in\big(0,\tfrac1{2\lambda_k}\big],
\end{align*}
with $C_2>0$ independent of~$k$. We proved at Step~3 that
\begin{align*}
\lim_{k\to\infty}\max_{\tau\in[0,T]}|\tilde l_{k1}''(\tau)|=0
\end{align*}
for each $T>0$. Consequently, by Lemma~\ref{IV_linearised},
\begin{align*}
\lim_{k\to\infty}\sup_{t\in (0,T]}\bigg(\frac{|s_{k2}(t)|}{t^2}+\frac{\left|s_{k2}'(t)\right|}{t}\bigg)=0.
\end{align*}
This means
$(\tilde l_{k2})_{k=1}^\infty$ converges to $\alpha_2t$ in the $C^2$-norm on $[0,T]$ as $k\to\infty$ with
\begin{align*}
\alpha_2=\lim_{k\to\infty}\alpha_{k2}\lambda_k.
\end{align*}
If $\alpha_2\ne0$, then
\begin{align*}
\lim_{k\to\infty}\tilde l_{k2}\big(\tfrac2{\alpha_2}\big)=2>1,
\end{align*}
which contradicts Lemma~\ref{so2_lem_max_princ}. Thus, $(\tilde l_{k1})_{k=1}^\infty$ converges to $0$ in the $C^2$-norm on $[0,T]$.

\vspace{5pt}\noindent\textit{Step 5.}
Assuming~\eqref{case1_so2_est} holds, fix $T>0$ such that 
\begin{align*}
T>\sup_{k\in\mathbb N}\frac{\tau_k}{\lambda_k}.
\end{align*}
Then
\begin{align*}
\max_{t\in[0,T]}\big(|\tilde l_{k1}'(t)|^2+|\tilde l_{k1}''(t)|+|\tilde l_{k2}'(t)|^2\big)
&\ge\big|\tilde l_{k1}'\big(\tfrac{\tau_k}{\lambda_k}\big)\big|^2+\big|\tilde l_{k1}''\big(\tfrac{\tau_k}{\lambda_k}\big)\big|+\big|\tilde l_{k2}'\big(\tfrac{\tau_k}{\lambda_k}\big)\big|^2
\\
&=\lambda_k^2\big(|l_{k1}'(\tau_k)|^2\big)+|l_{k1}''(\tau_k)|+|l_{k2}'(\tau_k)|^2=1
\end{align*}
for all~$k$. However, as we showed at Steps~3 and 4, $(\tilde l_{k1})_{k=1}^\infty$ must converge to $-1$ in the $C^2$-norm on~$[0,T]$, while $(\tilde l_{k2})_{k=1}^\infty$ converges to~$0$. This contradiction proves the lemma when~\eqref{case1_so2_est} holds.

\vspace{5pt}\noindent\textit{Step 6.}
Assume that
\begin{align}\label{case2_so2_est}
\lim_{k\to\infty}\frac{\tau_k}{\lambda_k}=\infty.    
\end{align}
Consider the functions $\bar{l}_{k1}$ and $\bar{l}_{k2}$ defined by
\begin{align*}
\bar{l}_{k1}(t)=l_{k1}(\tau_k+\lambda_kt),\qquad \bar{l}_{k2}(t)=l_{k2}(\tau_k+\lambda_kt),\qquad t\in\Big[-\frac{\tau_k}{\lambda_k},\frac{1-2\tau_k}{2\lambda_k}\Big].
\end{align*}
Clearly,
\begin{align*}
\max_{t\in\big[-\frac{\tau_k}{\lambda_k},\frac{1-2\tau_k}{2\lambda_k}\big]}\big(|\bar l_{k1}'(t)|^2&+|\bar l_{k1}''(t)|+|\bar l_{k2}'(t)|^2\big)  
\\
&=\lambda_k^2\max_{\tau\in[0,\frac12]}\big(|l_{k1}'(\tau)|^2+|l_{k1}''(\tau)|+|l_{k2}'(\tau)|^2\big)=1.
\end{align*}
Using~\eqref{smoothessphi}, we find
\begin{align*}
|\bar{l}_{k1}''(t)|&=\lambda_k^2|l_{k1}''(\tau)||_{\tau=\tau_k+\lambda_kt}
\\
&=\lambda_k^2\bigg|\frac{\phi_1'(\tau)}{\phi_1(\tau)}l_{k1}'(\tau)+\frac{\phi_1(\tau)}{\sigma(\tau)}l_{k1}(\tau)\big(pl_{k2}(\tau) +(1-p)\sin\tfrac{\pi\tau}{2}\big)^2\bigg|_{\tau=\tau_k+\lambda_kt}
\\
&\le\lambda_k\bigg|\frac{\phi_1'(\tau)}{\phi_1(\tau)}\bigg|_{\tau=\tau_k+\lambda_kt}+\lambda_k^2\bigg|\frac{\phi_1(\tau)}{\sigma(\tau)}\bigg|_{\tau=\tau_k+\lambda_kt}
\\
&\le\frac{\lambda_k}{\tau_k+\lambda_kt}+\lambda_k\sup_{\tau\in[0,\frac12]}|S_{\phi_1}(\tau)|+\frac{\lambda_k^2}{y_2(0)^2}+\lambda_k^2\sup_{\tau\in[0,\frac12]}|S_{\sigma1}(\tau)|
\end{align*}
and, similarly,
\begin{align*}
|\bar{l}_{k2}''(t)|
&\le\lambda_k\bigg|\frac{\phi_2'(\tau)}{\phi_2(\tau)}\bigg|_{\tau=\tau_k+\lambda_kt}+\lambda_k^2\bigg|\frac{\phi_2(\tau)}{\sigma(\tau)}\bigg|_{\tau=\tau_k+\lambda_kt}
\\
&\le\frac{\lambda_k}{\tau_k+\lambda_kt}+\lambda_k\sup_{\tau\in[0,\frac12]}|S_{\phi_2}(\tau)|+\frac{\lambda_k^2}{(\tau_k+\lambda_kt)^2}+\frac{\lambda_k^2\left|y_1'''(0)\right|}{3}+\lambda_k^2\sup_{\tau\in[0,\frac12]}|S_{\sigma2}(\tau)|
\end{align*}
on the interval $\big(-\frac{\tau_k}{\lambda_k},\frac{1-2\tau_k}{2\lambda_k}\big]$. These estimates and formula~\eqref{case2_so2_est} imply that $\bar l_{k1}''$ and $\bar l_{k2}''$ converge to 0 uniformly on $[-T,0]$ for each~$T>0$.

\vspace{5pt}\noindent\textit{Step 7.}
We have
\begin{align*}
|\bar l_{k1}'(0)|^2+|\bar l_{k1}''(0)|+|\bar l_{k2}'(0)|^2 =\lambda_k^2\big(|l_{k1}'(\tau_k)|^2+|l_{k1}''(\tau_k)|+|l_{k2}'(\tau_k)|^2\big)=1.
\end{align*}
Passing to subsequences if necessary, we may assume that
\begin{align*}
\lim_{k\to\infty}\bar l_{ki}'(0)=\eta_i\in[-1,1],\qquad i=1,2.
\end{align*}
The convergence of $\bar l_{k1}''$ to 0 implies that $\eta_1^2+\eta_2^2=1$. For each $T>0$,
\begin{align*}
\lim_{k\to\infty}\max_{t\in[-T,0]}|\bar l_{ki}'(t)-\eta_i|=\lim_{k\to\infty}\max_{t\in[-T,0]}|\bar l_{ki}'(t)-\bar l_{ki}'(0)|\le T\lim_{k\to\infty}\max_{t\in[-T,0]}|\bar l_{ki}''(t)|=0.
\end{align*}
If $T=\frac9{|\eta_1|+|\eta_2|}$ and $k$ is sufficiently large, then
\begin{align*}
|\bar l_{k1}(0)-\bar l_{k1}(-T)|&+|\bar l_{k2}(0)-\bar l_{k2}(-T)|
\\
&\ge T\Big(\min_{t\in[-T,0]}|\bar l_{k1}'(t)|+\min_{t\in[-T,0]}|\bar l_{k2}'(t)|\Big)
\ge\frac{T(|\eta_1|+|\eta_2|)}2>4.
\end{align*}
At the same time, by Lemma~\ref{so2_lem_max_princ},
\begin{align*}
|\bar l_{k1}(0)-\bar l_{k1}(-T)|+|\bar l_{k2}(0)-\bar l_{k2}(-T)|\le|\bar l_{k1}(0)|+|\bar l_{k1}(-T)|+|\bar l_{k2}(0)|+|\bar l_{k2}(-T)|\le4.
\end{align*}
This contradiction completes the proof.
\end{proof}

If $p=0$, system~\eqref{so2eqs_F} becomes 
\begin{align}\label{eqs_so2_p=0}
\bigg(\frac{l_1'}{\phi_1}\bigg)'=\frac{l_1\sin^2\frac{\pi t}{2}}{\sigma}, \qquad \bigg(\frac{l_2'}{\phi_2}\bigg)'=\frac{l_2\cos^2\frac{\pi t}{2}}{\sigma}.
\end{align}
In order to use Brouwer degree theory, we will need the following two results concerning solutions to equations~\eqref{eqs_so2_p=0}. The first one is a variant of the maximum principle for~\eqref{eqs_so2_p=0}.

\begin{lemma}\label{lem_mp_p=0}
Assume that $l_1$ and $l_2$ are continuous on $\big[0,\frac12\big]$ and satisfy~(\ref{eqs_so2_p=0}) on $\big(0,\frac12\big)$. The following statements hold:
\begin{enumerate}
\item
If $l_1(0)=-1$ and $l_1(\frac12)\le 0$, then $-1<l_1<0$ on $\big(0,\tfrac12\big)$.
\item If $l_2(0)=0$ and $l_2(\frac{1}{2})\le 1$, then $0<l_2<1$ on $\big(0,\frac{1}{2}\big)$.
\end{enumerate}
Similarly, assume that $l_1$ and $l_2$ are continuous on $\big[\frac12,1\big]$ and satisfy~(\ref{eqs_so2_p=0}) on $\big(\frac12,1\big)$. Then we have:
\begin{enumerate}
\item
If $l_1(\frac12)\ge -1$ and $l_1(1)=0$, then $-1<l_1<0$ on $\big(\tfrac12,1\big)$.
\item If $l_2(\frac{1}{2})\ge 0$ and $l_2(1)=1$, then $0<l_2<1$ on $\big(\frac{1}{2},1\big)$. 
\end{enumerate}
\end{lemma}

\begin{proof}
It suffices to repeat the proof of Lemma~\ref{so2_lem_max_princ} with minor modifications.
\end{proof}

Our next result is, in a sense, a refined version of Lemma~\ref{so2uniformbounds} that applies when $p=0$.

\begin{lemma}\label{lem_ok_big}
There exists a number $R_1>0$ such that $l_{1-}^{0,\alpha}(\frac12)>0$ if $\alpha_1\ge R_1$, $l_{1+}^{0,\beta}(\frac12)<-1$ if $\beta_1\ge R_1$, $l_{2-}^{0,\alpha}(\frac12)>1$ if $\alpha_2\ge R_1$, and $l_{2+}^{0,\beta}(\frac12)<0$ if $\beta_2\ge R_1$.
\end{lemma}

\begin{proof}
Let us begin with the estimate on $l_{1-}^{0,\alpha}(\frac12)$. Consider the function $s_1$ such that
\begin{align*}
l_{1-}^{0,\alpha}=s_1+\frac{\alpha_1 t^2}2-1.
\end{align*}
Assuming $\alpha_1>16$, denote
\begin{align*}
T=\sqrt{\frac4{\alpha_{1}}}<\tfrac12.
\end{align*}
We will show that $|s_1|<1$ on the interval $[0,T]$ if $\alpha_1$ is large. This would mean $l_{1-}^{0,\alpha}(T)>0$, and Lemma~\ref{lem_mp_p=0} would imply $l_{1-}^{0,\alpha}(\frac12)>0$.

Using~\eqref{eqs_so2_p=0}, one easily concludes that $s_1$ solves the first equation in~\eqref{lin_IVP} with 
\begin{align*}
    \mathcal{S}_1&=\frac1t\bigg((l_{1-}^{0,\alpha})''-\frac{(l_{1-}^{0,\alpha})'}{t}\bigg) \\
    &=\frac1t\bigg((s_1'+\alpha_1 t)\left(\frac{\phi_1'}{\phi_1}-\frac{1}{t}\right)+\frac{(s_1+\frac{\alpha_1 t^2}{2}-1)\phi_1 \sin^2\frac{\pi t}{2}}{\sigma}\bigg).
\end{align*}
From~\eqref{smoothessphi}, we find
\begin{align*}
|\mathcal S_1|&=\Big|\frac1t(s_1'+\alpha_1t)S_{\phi_1}+\frac1{2t}(2s_1+\alpha_1t^2-2)\Big(\frac1{y_2(0)^2}+S_{\sigma1}\Big)\mathcal \sin^2\tfrac{\pi t}2\Big|
\\
&\le C\Big(\frac{|s_1|}t+\alpha_1t+t+|s_1'|\Big),
\end{align*}
where $C>0$ is a constant independent of $\alpha_1$. By Lemma~\ref{IV_linearised}, we have
\begin{align*}
\sup_{t\in(0,T]}\Big(\frac{|s_1|}{t^2}&+\frac{|s_1'|}t\Big)
\le\tilde CT\sup_{t\in(0,T]}\Big(\frac{|s_1|}t+\alpha_1t+t+|s_1'|\Big)
\end{align*}
with $\tilde C>0$. Assume that $\alpha_1>\max\{16,256\tilde C^2\}$. Then
\begin{align*}
T=\sqrt{\frac4{\alpha_1}}<\min\big\{\tfrac12,(8\tilde C)^{-1}\big\},
\end{align*}
which implies
\begin{align*}
\sup_{t\in(0,T]}\Big(\frac{|s_1|}{t}+|s_1'|\Big)&\le \tilde CT^2\sup_{t\in(0,T]}\Big(\frac{|s_1|}t+|s_1'|\Big)+\tilde CT^3\alpha_1+\tilde CT^3
\\
&\le\tfrac1{16}\sup_{t\in(0,T]}\Big(\frac{|s_1|}t+|s_1'|\Big)+\tfrac{17}{32}.
\end{align*}
Ergo,
\begin{align*}
\sup_{t\in(0,T]}\Big(\frac{|s_1|}{t}+|s_1'|\Big)\le\tfrac{17}{30}<1.
\end{align*}
We easily conclude that $|s_1|<1$ and $l_{1-}^{0,\alpha}(\frac12)>0$. Analogous arguments yield the required estimates for $l_{2-}^{0,\alpha}(\frac12)$, $l_{1+}^{0,\beta}(\frac12)$ and $l_{2+}^{0,\beta}(\frac12)$.
\end{proof}

\subsection{Existence for the global problem}

Our next result explains the structure of the map~$G_0$.

\begin{lemma}
There exist non-degenerate affine maps $H_1,H_2:\mathbb R^2\to\mathbb R^2$ such that
\begin{align}\label{linearmapp=0}
G_0(\alpha_1,\alpha_2,\beta_1,\beta_2)=(H_1(\alpha_1,\beta_1),H_2(\alpha_2,\beta_2)),\qquad (\alpha_1,\alpha_2,\beta_1,\beta_2)\in\mathbb R^4.
\end{align}
\end{lemma}

\begin{proof}
For $i=1,2$, define
\begin{align*}
H_i(x,y)&=A_i(x,y)+B_i,\qquad x,y\in\mathbb R.
\end{align*}
On the right-hand side, $A_i:\mathbb R^2\to\mathbb R^2$ and $B_i\in\mathbb R^2$ are given by
\begin{align*}
A_i(\alpha_i,\beta_i)&=\big(\big(l_{i+}^{0,\beta}-l_{i+}^{0,0}-l_{i-}^{0,\alpha}+l_{i-}^{0,0}\big)(\tfrac12),\big(l_{i+}^{0,\beta}-l_{i+}^{0,0}-l_{i-}^{0,\alpha}+l_{i-}^{0,0}\big)'(\tfrac12)\big),
\\ 
B_i&=\big(\big(l_{i+}^{0,0}-l_{i-}^{0,0}\big)(\tfrac12),\big(l_{i+}^{0,0}-l_{i-}^{0,0}\big)'(\tfrac12)\big),
\end{align*}
where $\alpha\in \mathbb{R}^2$ and $\beta\in \mathbb{R}^2$ are vectors with $i$th components equal to $\alpha_i$ and $\beta_i$, respectively. It is easy to see that $A_i(\alpha_i,\beta_i)$ does not depend on the choice of such $\alpha$ and~$\beta$. Equality~\eqref{linearmapp=0} is evident. Let us prove that $H_1$ and $H_2$ are affine. One can easily check that $A_1$ and $A_2$ are linear. It remains to show that $A_1$ and $A_2$ are nondegenerate.

The first equation in~\eqref{eqs_so2_p=0} and the condition $\big(l_{1-}^{0,0}\big)''(0)=0$ yield
\begin{align*}
\big(l_{1-}^{0,0}\big)'(t)=\phi_1\int_0^t\frac{l_{1-}^{0,0}(\tau)}{\sigma(\tau)}\sin^2\tfrac{\pi\tau}2\,d\tau,\qquad t\in\big(0,\tfrac12\big].
\end{align*}
Since $l_{1-}^{0,0}(0)=-1$, we can conclude that 
\begin{align}\label{1st_quadr}
l_{1-}^{0,0}(\tfrac12)<-1,\qquad \big(l_{1-}^{0,0}\big)'(\tfrac12)<0.
\end{align}
Lemma~\ref{lem_ok_big} and the intermediate value theorem imply the existence of $\alpha_*\in(0,R_0)$ with $l_{1-}^{0,(\alpha_*,0)}(\tfrac12)=0$. The function $l_{1+}^{0,0}$ is identically~0. Using Lemma~\ref{lem_ok_big} and the intermediate value theorem again, we obtain $\beta_*\in(0,R_0)$ such that $l_{1+}^{0,(\beta_*,0)}(\tfrac12)=-1$. To prove that $A_1$ is nondegenerate, it suffices to show that the determinant of the matrix with rows $A_1(\alpha_*,0)$ and $A_1(0,\beta_*)$ does not vanish.

We compute
\begin{align*}
\det&\begin{pmatrix}
A_1(\alpha_*,0) \\
A_1(0,\beta_*)
\end{pmatrix}
=\det\begin{pmatrix}\big(l_{1-}^{0,0}-l_{1-}^{0,(\alpha_*,0)}\big)(\tfrac12) & \big(l_{1-}^{0,0}-l_{1-}^{0,(\alpha_*,0)}\big)'(\tfrac12)
\\
\big(l_{1+}^{0,(\beta_*,0)}-l_{1+}^{0,0}\big)(\tfrac12) & \big(l_{1+}^{0,(\beta_*,0)}-l_{1+}^{0,0}\big)'(\tfrac12)
\end{pmatrix}
\\
&=
\det\begin{pmatrix}l_{1-}^{0,0}(\tfrac12) & \big(l_{1-}^{0,0}-l_{1-}^{0,(\alpha_*,0)}\big)'(\tfrac12)
\\
-1 & \big(l_{1+}^{0,(\beta_*,0)}\big)'(\tfrac12)
\end{pmatrix}
=l_{1-}^{0,0}(\tfrac12)\big(l_{1+}^{0,(\beta_*,0)}\big)'(\tfrac12)+\big(l_{1-}^{0,0}\big)'(\tfrac12)-\big(l_{1-}^{0,(\alpha_*,0)}\big)'(\tfrac12).
\end{align*}
Lemma~\ref{lem_mp_p=0} implies that $\big(l_{1+}^{0,(\beta_*,0)}\big)'(\tfrac12)\ge0$ and  $\big(l_{1-}^{0,(\alpha_*,0)}\big)'(\tfrac12)\ge0$. Together with~\eqref{1st_quadr}, these inequalities yield
\begin{align*}
\det&\begin{pmatrix}
A_1(\alpha_*,0) \\
A_1(0,\beta_*)
\end{pmatrix}
<0.
\end{align*}
Thus, $A_1$ is nondegenerate. Analogous reasoning works on $A_2$.
\end{proof}

We are now ready to demonstrate that the function $G_1$ has a zero on $\mathbb R^4$. As explained above, this would imply the existence of~$l_1$ and~$l_2$ satisfying~\eqref{so2equationsli}--\eqref{so2smoothnessli}.

\begin{lemma}\label{lem_degree_nonzero}
There is a 4-tuple $(\hat\alpha_{1},\hat\alpha_{2},\hat\beta_{1},\hat\beta_{2})\in\mathbb R^4$ such that
$G_1(\hat\alpha_{1},\hat\alpha_{2},\hat\beta_{1},\hat\beta_{2})=0$.
\end{lemma}

\begin{proof}
Fix a number $R$ greater than $2\max\{R_0,R_1\}$, where $R_0$ and $R_1$ come from Lemmas~\ref{so2uniformbounds} and~\ref{lem_ok_big}. Recall that $\Omega_R$ stands for the open disc in $\mathbb R^2$ of radius $R$ centred at the origin. We will show that the Brouwer degree of $G_1$ in $\Omega_R\times\Omega_R$, denoted $\deg(G_1,\Omega_R\times\Omega_R)$, does not vanish. This would mean that $G_1$ has a zero in~$\Omega_R\times\Omega_R$; see~\cite[Theorem 2.2]{SDT}.

Homotopy invariance of the Brouwer degree implies
\begin{align*}
\deg(G_1,\Omega_R)=\deg(G_0,\Omega_R).
\end{align*}
One easily checks that
\begin{align*}
\deg(G_0,\Omega_R\times\Omega_R)=\deg(H_1,\Omega_R)\deg(H_2,\Omega_R).
\end{align*}
Thus, it suffices to show that $\deg(H_1,\Omega_R)\ne0$ and $\deg(H_2,\Omega_R)\ne0$.

The boundary $\partial\Omega_R$ is a circle of radius~$R$. The degree $\deg(H_1,\Omega_R)$ is equal to the winding number of the curve $H_1(\partial\Omega_R)$ around the origin; see~\cite[Remark~1.2.2]{DincaMawhin}. Because $H_1$ is affine, this curve is an ellipse. Enlarging $R$ if necessary, we may assume that it encloses the origin. Therefore, $\deg(H_1,\Omega_R)$ is either $1$ or $-1$. Analogous reasoning shows that $\deg(H_2,\Omega_R)$ does not vanish.
\end{proof}

Lemma~\ref{lem_degree_nonzero} implies the existence of $C^2$-differentiable functions $l_1$ and $l_2$ on $[0,1]$ that solve~\eqref{so2equationsli}--\eqref{so2smoothnessli}. These functions are given by~\eqref{sol_formula}. We will show that they possess properties~1--3 from Lemma~\ref{so2equationslemma}. This will imply the existence of a metric $g$ such that $X(g)=Y$.

\begin{proof}[Proof of Theorem \ref{so2existence}]
\noindent\textit{Step 1}.
We begin by establishing the smoothness on $[0,1)$ of the functions $l_1$ and $l_2$ given by~\eqref{sol_formula}. Define
\begin{align*}
s_1=l_1-\frac{\hat\alpha_1t^2}2+1,\qquad s_2=l_2-\hat\alpha_2t. 
\end{align*}
Obviously, $s_1$ and $s_2$ are $C^2$-differentiable on~$[0,1)$. We use induction to show that they are, in fact, smooth on this interval. It is easy to see that
\begin{align*}
s_1''-\frac{s_1'}{t}&=S_{\phi_1}l_1'+\Big(\frac{1}{y_2(0)^2}+S_{\sigma_1}\Big)l_1l_2^2, \\
s_2''+\frac{s_2'}{t}-\frac{s_2}{t^2}&=S_{\phi_2}l_2'+\frac{(l_1^2-1)l_2}{t^2}+\Big(S_{\sigma_2}
-\frac{y_1'''(0)}{3}\Big)l_1^2l_2.
\end{align*}
If $s_1$ and $s_2$ are $C^k$-differentiable on $[0,1)$ for some~$k\ge2$, then they must be $C^{k+1}$-differentiable by Lemmas~\ref{DDT} and~\ref{IV_linearised}. We conclude that $s_1$ and $s_2$ are smooth on~$[0,1)$. Obviously, the same can be said about $l_1$ and $l_2$.

\vspace{5pt}\noindent\textit{Step 2}. Let us prove that $l_1$ and $l_2$ extend to a smooth even function and a smooth odd function on~$(-1,1)$, respectively. It suffices to show that the derivatives $s_1^{(2k+1)}(0)$ and $s_2^{(2k)}(0)$ vanish for each $k\in[0,\infty)\cap\mathbb Z$. We proceed by induction. Thanks to~\eqref{so2smoothnessli}, $s_1'(0)=s_2(0)=0$. Assume that, for some~$k$,
\begin{align*}
s_1'''(0)&=s_1^{(5)}(0)=\cdots=s_1^{(2k+1)}(0)=0, \\
s_2''(0)&=s_2^{(4)}(0)=\cdots=s_2^{(2k)}(0)=0.
\end{align*}
Taylor's theorem, the above ODE for $s_1$ and conditions~\eqref{so2smoothnessli} imply
\begin{align*}
\sum_{n=0}^{2k+1}\frac{s_1^{(n+2)}(0)t^{n}}{n!}&+\mathcal U''(t) -\sum_{n=0}^{2k+1}\frac{s_1^{(n+2)}(0)t^{n}}{(n+1)!}-\frac{\mathcal U'(t)}t
\\ &=\sum_{n=0}^{2k+1}\Big(S_{\phi_1}l_1'+\Big(\frac{1}{y_2(0)^2}+S_{\sigma_1}\Big)l_1l_2^2\Big)^{(n)}(0)\frac{t^n}{n!}+\mathcal V(t),
\end{align*}
where $\mathcal U(t)=o(t^{2k+3})$ and $\mathcal V(t)=o(t^{2k+1})$ as $t\to0$. Equating the coefficients at $t^{2k+1}$, we obtain
\begin{align}\label{parity_big}
\frac{2k+1}{2k+2}s_1^{(2k+3)}(0)
&=\Big(S_{\phi_1}l_1'+\Big(\frac{1}{y_2(0)^2}+S_{\sigma_1}\Big)l_1l_2^2\Big)^{(2k+1)}(0) \notag
\\
&=\sum_{n=0}^{2k+1}{2k+1 \choose n}S_{\phi_1}^{(n)}(0)l_1^{(2k+2-n)}(0)+\frac{\big(l_1l_2^2\big)^{(2k+1)}(0)}{y_2(0)^2} \notag
\\
&\hphantom{=}~+\sum_{n=0}^{2k+1}{2k+1 \choose n}S_{\sigma_1}^{(n)}(0)\big(l_1l_2^2\big)^{(2k+1-n)}(0).
\end{align}
The functions $\phi_1$ and $\sigma$ extend to smooth odd functions on $(-1,1)$. Therefore,
\begin{align*}
S_{\phi_1}(0)&=S_{\phi_1}''(0)=S_{\phi_1}^{(4)}(0)=\cdots=S_{\phi_1}^{(2k)}(0)=0,
\\
S_{\sigma_1}'(0)&=S_{\sigma_1}'''(0)=S_{\sigma_1}^{(5)}(0)=\cdots=S_{\sigma_1}^{(2k+1)}(0)=0.
\end{align*}
These formulas and the induction hypothesis imply that the expression in the last two lines of~\eqref{parity_big} vanishes. We conclude that $s_1^{(2k+3)}(0)=0$. Analogous, though somewhat more tedious, reasoning shows that $s_2^{(2k+2)}(0)=0$. Thus, $l_1$ and $l_2$ extend to a smooth even function and a smooth odd function on $(-1,1)$, respectively.

\vspace{5pt}\noindent\textit{Step 3}. Our objective is to prove that $l_1$ and $l_2$ possess properties~1--3 from Lemma~\ref{so2equationslemma}. Arguing as in Step~2, we can show that $l_1$ ($l_2$) extends to a smooth function on $(0,2)$ that is odd (even) about $t=1$. Thus, $l_1$ and $l_2$ possess properties~1 and~2. Lemma~\ref{so2_lem_max_princ} implies that $l_2'(0)\ge0$ and $l_1'(1)\ge0$. If $l_2'(0)=0$, then $l_2$ identically zero on $[0,1]$, which contradicts~\eqref{so2smoothnessli}. We conclude that $l_2'(0)>0$. Analogously, $l_1'(1)>0$. Thus, $l_1$ and $l_2$ possess property~3. Now Lemma~\ref{so2equationslemma} implies the existence of an $SO(2)\times SO(2)$-invariant metric $g$ such that~$X(g)=Y$.
\end{proof}

\section{Existence for $SO(3)$-invariant metrics}\label{SectionexistenceSO3}

We can identify $\mathbb S^3$ with the one-point compactification of $\mathbb R^3$ via the stereographic projection. Together with the canonical action of $SO(3)$ on $\mathbb R^3$, this identification yields a cohomogeneity one action of $SO(3)$ on~$\mathbb S^3$. The principal orbits are spheres~$\mathbb S^2$; there are two singular orbits. Let $Y$ be a symmetric positive-definite $SO(3)$-invariant tensor field on~$\mathbb S^3$. We can view $Y$ as a Riemannian metric. Fix a $Y$-geodesic $\gamma:[-1,1]\to\mathbb S^3$ connecting the singular orbits of the $SO(3)$ action. This gives us a natural diffeomorphism between the principal part of $\mathbb S^3$ and the product~$(-1,1)\times \mathbb{S}^2$. Let $Q$ be the round metric on $\mathbb{S}^2$ with scalar curvature~$2$. Scaling $Y$ if necessary, we can find a smooth function $y$ on $(-1,1)$ such that
\begin{align}\label{s03y}
Y=dr\otimes dr+y(r)^2Q,\qquad r\in(-1,1).
\end{align}
We are ready to state the main result of this section.

\begin{theorem}\label{so3exist}
Let $Y$ be a symmetric positive-definite $SO(3)$-invariant tensor field on~$\mathbb S^3$ given by~(\ref{s03y}). There exists an $SO(3)$-invariant Riemannian metric $g$ such that~$X(g)=Y$.
\end{theorem}

The rest of this section is devoted to the proof of Theorem~\ref{so3exist}.

\subsection{The equations}\label{initialequations}

Suppose that $h$ and $f$ are smooth and positive functions on $(-1,1)$. Then
\begin{align}\label{so3metric}
g=h(r)^2dr\otimes dr+f(r)^2Q
\end{align}
is an $SO(3)$-invariant metric on $(-1,1)\times\mathbb S^2$. The following well-known lemma tells us when it can be extended to~$\mathbb S^3$. 

\begin{lemma}\label{so3smoothness}
Formula~(\ref{so3metric}) defines an $SO(3)$-invariant metric on $\mathbb S^3$ if and only if the following requirements are satisfied:
\begin{enumerate}
\item 
One can extend $h$ to a smooth function on $(-3,3)$ that is even about $r=-1$ and about $r=1$.
\item
One can extend $f$ to a smooth function on $(-3,3)$ that is odd about $r=-1$ and about $r=1$.
\item
The equalities $f'(-1)=h(-1)>0$ and $f'(1)=-h(1)<0$ hold.
\end{enumerate}
\end{lemma}

The proof is similar to that of Lemma~\ref{lem_so2_smooth}; see~\cite[page~65]{RH84}.

Applying Lemma~\ref{so3smoothness} to the tensor field $Y$, we arrive at the following conclusions:
\begin{enumerate}
\item
One can extend $y$ to a smooth function on $(-3,3)$ that is odd about $r=-1$ and about $r=1$.
\item
The formula $y'(-1)=-y'(1)=1$ holds.
\end{enumerate}

Let us convert the equation $X(g)=Y$ to an ODE. Given a metric $g$ on $\mathbb S^3$ of the form~\eqref{so3metric}, define 
\begin{align}\label{lsigmaso3}
l=-\frac{f'}{h},\qquad \sigma=y^2. 
\end{align}
Clearly, $l$ can be extended to a smooth function on $(-3,3)$ that is even about $r=-1$ and about $r=1$.

\begin{lemma}\label{lem_so3_smth}
If a metric $g$ on $\mathbb S^3$ of the form~\eqref{so3metric} is such that $X(g)=Y$, then $l$ solves the ODE
\begin{align}\label{MEFL'}
l''=\frac{l^3-l}{\sigma},\qquad r\in(-1,1),
\end{align} 
subject to the boundary conditions
\begin{equation}\label{SCSO3}
l(-1)=-1, \qquad l(1)=1.
\end{equation}
Moreover, $l'>0$ for $r\in(-1,1)$, $l''(-1)>0$ and $l''(1)<0$. Conversely, suppose one can find $l$ that satisfies~(\ref{MEFL'})--(\ref{SCSO3}) and possesses the following properties:
\begin{enumerate}
\item
Onc can extend $l$ to a smooth function on $(-3,3)$ that is even about $r=-1$ and about $r=1$.
\item
The inequality $l'>0$ holds on $(-1,1)$. Also, $l''(-1)>0$ and $l''(1)<0$.
\end{enumerate}
Then there exists an $SO(3)$-invariant metric $g$ on $\mathbb S^3$ such that $X(g)=Y$.
\end{lemma}

\begin{proof}
If $g$ is given by~\eqref{so3metric}, then
\begin{align*}
\Ric(g)&=-2\left(\frac{f''}{f}-\frac{h'f'}{hf}\right)dr\otimes dr+\left(-\frac{f''f}{h^2}+\frac{fh'f'}{h^3}-\frac{(f')^2}{h^2}+1\right)Q 
\\
&=\frac{2hl'}{f}\,dr\otimes dr+\left(\frac{fl'}{h}-l^2+1\right)Q,
\\
S(g)&=\frac{4l'}{hf}-\frac{2(l^2-1)}{f^2};
\end{align*}
see, e.g.,~\cite[page~67]{RH84} and~\cite[Lemma~3.1]{Pulemotov16}. Consequently,
\begin{align*}
\Ein(g)&=\frac{h^2(l^2-1)}{f^2}\,dr\otimes dr-\frac{fl'}{h}Q,
\\
X(g)&=\frac{(l')^2}{f^2}\,dr\otimes dr-\frac{l'(l^2-1)}{hf}Q.
\end{align*}
If $X(g)=Y$, then
\begin{align*}
l'=\pm f, \qquad l^2-1=\mp y^2h. 
\end{align*}
Lemma~\ref{so3smoothness} shows that conditions~\eqref{SCSO3} hold for~$l$. Since $f$ is assumed positive, these conditions imply $l'=f>0$ on $(-1,1)$. This means
\begin{align*}
    l''=f'=-lh=\frac{l(l^2-1)}{y^2}=\frac{l^3-l}{\sigma}.
\end{align*}
Conversely, suppose $l$ satisfies~\eqref{MEFL'}--\eqref{SCSO3} and possesses properties~1--2 from the statement of the lemma. Define
\begin{align*}
f=l',\qquad h=\frac{1-l^2}{y^2}.
\end{align*}
It is easy to see that the metric $g$ given by~\eqref{so3metric} solves $X(g)=Y$.
\end{proof}

\subsection{The linear problem and existence near the singular orbits}

Choose a function $\zeta_p:[-1,1]\to(0,\infty)$ for every~$p\in[0,1]$ so that the map
\begin{align*}
[-1,1]\times[0,1]\ni(r,p)\mapsto\zeta_p(r)\in(0,\infty)
\end{align*}
is smooth. Assume that
\begin{align}\label{bc_zeta_p}
\zeta_p(-1)=\zeta_p'(1)=\zeta_p'''(-1)=\zeta_p'''(1)=0,\qquad \zeta_p''(-1)=\zeta_p''(1)=2,\qquad p\in[0,1].
\end{align}
Formulas~\eqref{bc_zeta_p} imply
\begin{align}\label{zeta_p_asy}
\frac1{\zeta_p}=\frac1{(r+1)^2}+S_{\zeta_p},\qquad r\in(-1,1),
\end{align}
with $S_{\zeta_p}$ smooth on $[-1,1)$. Our plan is to prove the solvability of the ODE
\begin{align}\label{eq_zeta}
l''=\frac{l^3-l}{\zeta_p},\qquad r\in(-1,1),~p\in[0,1],
\end{align}
subject to the boundary conditions~\eqref{SCSO3}. Afterwards, we will use Brouwer degree theory to show that a solution possessing properties~1--2 from Lemma~\ref{lem_so3_smth} exists if $\zeta_p=\sigma$.

Let us begin by studying the linearised version of equation~\eqref{eq_zeta}.

\begin{lemma}\label{so3linearexistence}
Consider a continuous function $S$ on $[-1,-1+T]$ for some $T>0$. Given $k\in[0,\infty)\cap\mathbb Z$, assume that $(1+t)S$ is $C^k$-differentiable on $[-1,-1+T]$. There exists a unique $C^{k+2}$ function $s:[-1,-1+T]\to\mathbb R$ such that
\begin{align*}
 s''=\frac{2s}{(1+t)^2}+(1+t)S,\qquad s(-1)=s'(-1)=s''(-1)=0.
\end{align*}
Moreover,
 \begin{align}\label{so3linearestimates}
   \sup_{t\in(-1,-1+T]} \frac{|s(t)|}{(1+t)^2}\le \frac T4\sup_{t\in(-1,-1+T]}|S(t)|. 
 \end{align}
\end{lemma}

\begin{proof}
The ODE for $s$ is equivalent to
\begin{align*}
\bigg((1+t)^4\bigg(\frac{s}{(1+t)^2}\bigg)'\bigg)'=(1+t)^3S.
\end{align*}
Solving explicitly and using the conditions at $t=-1$, we find
\begin{align*}
s(t)=(1+t)^2\int_{-1}^t\int_{-1}^\tau\frac{(1+\rho)^3}{(1+\tau)^4}S(\rho)\,d\rho\,d\tau.
\end{align*}
Clearly, \eqref{so3linearestimates} holds. Let us check the regularity of~$s$. It follows from~\eqref{so3linearestimates} that $s$ is $C^2$-differentiable. Therefore, we may assume that $k\ge1$. By Lemma~\ref{IIOR}, the expression 
\begin{align*}
R(t)=\frac{1}{(1+t)^2}\int_{-1}^{t}(1+\rho)^3S(\rho)d\rho
\end{align*}
defines a $C^{k+1}$ function on $[-1,-1+T]$. The continuity of $S$ implies
$$
\lim_{t\to-1}\frac{R(t)}{1+t}=R(-1)=0.
$$
Consequently, by Lemma~\ref{DDT}, the expression $\frac{R(t)}{(1+t)^2}$ defines a $C^{k-1}$ function on $[-1,-1+T]$. We compute
\begin{align*}
s''(t)=2\int_{-1}^{t}\frac{R(\tau)}{(1+\tau)^2} +
2\frac{R(t)}{(1+t)}+R'(t).
\end{align*}
Another application of Lemma~\ref{DDT} enables us to conclude that $s$ is $C^{k+2}$-differentiable.
\end{proof}

Next, we prove a short-time existence result for~\eqref{eq_zeta}.

\begin{lemma}\label{LocalExistence}
Given $\alpha\in\mathbb{R}$, equation~\eqref{eq_zeta} has a unique solution on $(-1,-1+T]$ for some $T>0$ with
\begin{equation}\label{ICL''-1}
l(-1)=-1, \qquad l'(-1)=0,\qquad l''(-1)=\alpha.
\end{equation}
The values of $l$ and $l'$ at every $t\in(-1,-1+T]$ depend continuously on~$\alpha$ and~$p$. Similarly, given $\beta\in\mathbb{R}$,~\eqref{eq_zeta} has a unique solution on $[1-T,1)$ for some $T>0$ with
\begin{equation}\label{ICL''1}
l(1)=1, \qquad l'(1)=0,\qquad l''(1)=-\beta.
\end{equation}
The values of $l$ and $l'$ at every $t\in[1-T,1)$ depend continuously on~$\beta$ and~$p$.
\end{lemma}

\begin{proof}
We begin by proving existence near $t=-1$. Formula~\eqref{zeta_p_asy} implies that the ODE
\begin{align*}
s''&=\frac{2s}{\zeta_p}+\frac{s^2(s-3)}{\zeta_p}+\frac{\alpha^3(t+1)^6}{8\zeta_p} \notag \\ &\hphantom{=}~+\frac{3\alpha^2(t+1)^4(s-1)}{4\zeta_p}+\frac{3\alpha(t+1)^2s(s-2)}{2\zeta_p}+\frac{\alpha(t+1)^2}{\zeta_p}-\alpha
\end{align*}
is equivalent to
\begin{align}\label{eq_s_SO3}
s''=\frac{2s}{(t+1)^2}+(t+1)\mathcal S\Big(t,\frac{s}{(t+1)^2},\alpha,p\Big),
\end{align}
where $\mathcal S:\mathbb R^4\to\mathbb R$ is a smooth function. One can recover solutions to~\eqref{eq_zeta} from those to~\eqref{eq_s_SO3} by setting
\begin{align*}
l=s+\frac{\alpha(t+1)^2}2-1. 
\end{align*}
Thus, it suffices to prove the existence of $s$ satisfying~\eqref{eq_s_SO3} with the initial conditions
\begin{align}\label{ini_hom_so3}
s''(-1)=s'(-1)=s(-1)=0.
\end{align}

Choose $T\in(0,1)$. Let $\mathfrak C$ stand for the completion of the space of smooth compactly supported functions on $(-1,-1+T]$ with respect to the norm
\begin{align*}
\|u\|_{\mathfrak C}=\sup_{t\in (-1,-1+T]}\frac{|u(t)|}{(t+1)^2}.
\end{align*}
Denote by $\mathfrak C_1$ the unit ball in $\mathfrak C$ centred at~0. We will apply the Banach fixed point theorem in $\mathfrak C_1$ to prove the solvability of~\eqref{eq_s_SO3}. Consider the map $L$ taking $u\in\mathfrak C$ to the unique $C^2$ function $x$ on $[-1,-1+T]$ such that
\begin{align*}
x''=\frac{2x}{(t+1)^2}+(t+1)\mathcal S\Big(t,\frac{u}{(t+1)^2},\alpha,p\Big), \qquad x''(-1)=x'(-1)=x(-1)=0.
\end{align*}
Lemma~\ref{so3linearexistence} guarantees the existence of~$x$. It also implies that $L$ is a contraction on~$\mathfrak C_1$ if $T$ is sufficiently small. Thus, $L$ has a fixed point $s$ satisfying~\eqref{eq_s_SO3}--\eqref{ini_hom_so3}. We conclude that equation~\eqref{eq_zeta} has a solution $l$ near $t=-1$ such that~\eqref{ICL''-1} holds. It remains to verify uniqueness and continuous dependence on $\alpha$ and~$p$. One can do this using inequality~\eqref{so3linearestimates} as in the proof of Lemma~\ref{IVPWD}. 

Similar arguments work near $t=1$.
\end{proof}

To extend $l$ constructed in Lemma~\ref{LocalExistence} to a larger interval, we consider a slight modification of equation~\eqref{eq_zeta}. Namely, suppose $E$ is a smooth function, nonpositive on $(-\infty,-1)$ and nonnegative on $(1,\infty)$, such that $E(x)=x^3-x$ if $|x|\le2$, $|E(x)|\le |x^3-x|$ if $2\le|x|\le3$, and $E(x)=0$ if $|x|\ge3$. Obviously, $|E|\le24$ on~$\mathbb R$. Instead of~\eqref{eq_zeta}, we will consider the equation
\begin{align}\label{MEFL}
l''=\frac{E(l)}{\zeta_p}.
\end{align}
Lemma~\ref{LocalExistence} and the boundedness of $E$ imply that~\eqref{MEFL} has a solution on $(-1,1)$ subject to condition~\eqref{ICL''-1}. This solution is unique.

\subsection{A particular choice of curvature}\label{subsec_sp_curv}

Our next goal is to prove the existence of a function $\sigma_0$ with a series of specific properties. We will use $\sigma_0$ as our ``starting point" when we apply Brouwer degree theory. In this subsection, $l$~denotes the solution to the initial-value problem
\begin{align}\label{eq_sigma_0}
l''=\frac{E(l)}{\sigma_0},\qquad l(-1)=-1,\qquad l'(-1)=0,\qquad l''(-1)=\alpha,
\end{align}
on $[-1,1)$.

\begin{lemma}\label{so3specialcurvature}
There exists a smooth function $\sigma_0:[-1,1]\to(0,\infty)$ possessing the following properties:
\begin{enumerate}
\item
The equality $\sigma_0(r)=(r+1)^2$ holds for all $r\in\left[-1,-\frac78\right]$.

\item
The function $\sigma_0$ is even about $r=0$.

\item
If $\alpha>0$, then the solution $l$ to the initial-value problem~(\ref{eq_sigma_0}) has positive derivative on~$(-1,0]$.

\item There exists $\alpha_0>0$ such that $l\left(-\frac78\right)>1$ if $\alpha\ge\alpha_0$.

\item If $l\big(\frac78\big)=0$, then $l'\big(\frac78\big)<\frac57$. 
\end{enumerate}
\end{lemma}

\begin{proof}
\textit{Step 1.} Denote by $\tilde l$ the solution to the initial-value problem
\begin{align}\label{eq_tildel_sp_cu}
\tilde l''=\frac{E(\tilde l)}{(r+1)^2},\qquad \tilde l(-1)=-1,\qquad \tilde l'(-1)=0,\qquad \tilde l''(-1)=1,
\end{align}
on $[-1,\infty)$. At Steps 1--3, we will show that $\tilde l'>0$ except at $r=-1$. To this end, define
$$
r_0=\sup\{\delta\ge-1\,|\,\tilde l(r)\in[-1,0]~\mbox{for}~r\in[-1,\delta]\}.
$$
The initial conditions in~\eqref{eq_tildel_sp_cu} imply that $r_0>-1$. The derivatives $\tilde l'$ and $\tilde l''$ are positive on $(-1,r_0)$. It is easy to see that $r_0<\infty$ and $\tilde l(r_0)=0$.
\vspace{5pt}

\noindent\textit{Step 2.} Consider the function
\begin{align*}
x(t)=\tilde l(e^t-1),\qquad t\in\mathbb R.
\end{align*}
Formulas~\eqref{eq_tildel_sp_cu} imply that $x$ satisfies
\begin{align*}
x''-x'-E(x)=0,\qquad \lim_{t\to-\infty}x(t)=-1,\qquad \lim_{t\to-\infty}x'(t)=0.
\end{align*}
Setting $t_0=\ln(r_0+1)$, we conclude that $x'$ and $x''$ are positive on $(-\infty,t_0)$ and $x(t_0)=0$. Moreover,
\begin{align*}
(x'-x)'=E(x)=x^3-x>0
\end{align*}
on this interval, and
\begin{align*}
\lim_{t\to-\infty}(x'-x)(t)=1.
\end{align*}
As a consequence, $$x'(t_0)>1+x(t_0)=1.$$
We will prove that $x'$ stays above 1 as long as $x$ is below~1. This will help us show that~$\tilde l'>0$.
\vspace{5pt}

\noindent\textit{Step 3.} Define
$$
t_1=\sup\{\delta\ge t_0\,|\,x(t)\in[0,1]~\mbox{for}~t\in[t_0,\delta]\}.
$$
Since $x(t_0)=0$ and $x'(t_0)>1$, this supremum is greater than~$t_0$. On the interval $(t_0,t_1)$,
\begin{align*}
x''=x'+E(x)\ge x'-1,
\end{align*}
which means $x'(t)\ge x'(t_0)>1$. It follows that $t_1<\infty$, $x(t_1)=1$ and $x'(t_1)>1$. Ergo, $\tilde l'>0$ on $[r_0,r_1]$, where $r_1=e^{t_1}-1$. Noting that $\tilde l(r_1)=x(t_1)=1$, one then checks easily that $\tilde l'>0$ on $(r_1,\infty)$. There exists a point $r_*\in(r_1,\infty)$ such that $\tilde l(r_*)>1$.
\vspace{5pt}

\noindent\textit{Step 4.} We are now ready to construct a function $\sigma_0$ with the properties listed in the lemma. Define
\begin{align*}
\alpha_0=64(r_*+1)^2.
\end{align*}
Choose a number $\alpha_-\in(0,\alpha_0)$ so that
\begin{align*}
\tilde l\Big(\frac{\sqrt{\alpha_-}}8-1\Big)<-\tfrac34, \qquad
\sqrt{\alpha_-}\,\sup\Big\{\tilde l'(r)\,\Big|\,r\in\Big[-1,\frac{\sqrt{\alpha_-}}8-1\Big]\Big\}<\tfrac14.
\end{align*}
The initial conditions in~\eqref{eq_tildel_sp_cu} imply that such a number exists. Denote
\begin{align*}
l_*=\sqrt{\alpha_-}\,\inf\Big\{\tilde l'(r)\,\Big|\,r\in\Big[\frac{\sqrt{\alpha_-}}8-1,\frac{\sqrt{\alpha_0}}8-1\Big]\Big\}.
\end{align*}
Let $\sigma_0$ be an even function on $(-1,1)$ such that $\sigma_0(r)=(r+1)^2$ for $r\in\left[-1,-\frac78\right]$ and
$$
\int_{-\frac78}^{\frac78}\frac1{\sigma_0(r)}\,dr<\min\Big\{\frac{l_*}{48},\frac1{192}\Big\}.
$$
Obviously, $\sigma_0$ possesses properties~1 and~2 listed in the lemma. Assume that $\alpha>0$. The solution $l$ to~\eqref{eq_sigma_0} satisfies 
\begin{align}\label{sp_cu_scaling}
l(r)=\tilde l(\sqrt{\alpha}(r+1)-1),\qquad r\in\left[-1,-\tfrac78\right].
\end{align}
If $\alpha\ge\alpha_0$, then
$$
l\left(-\tfrac78\right)\ge\tilde l\Big(\frac{\sqrt{\alpha_0}}8-1\Big)=\tilde l(r_*)>1.
$$
Thus, $\sigma_0$ possesses property~4.
\vspace{5pt}

\noindent\textit{Step 5.}
Formula~\eqref{sp_cu_scaling} implies that $l'>0$ on $\left[-1,-\tfrac78\right]$. Let us show that this inequality holds on $\left(-\tfrac78,\tfrac78\right]$. If $\alpha>\alpha_0$, the claim is evident since $E(l(r))$ is nonnegative for $r\in\left(-\tfrac78,\tfrac78\right]$. If $\alpha\in[\alpha_-,\alpha_0]$, then
$$
l'\left(-\tfrac78\right)=\sqrt{\alpha}\,\tilde l'\Big(\frac{\sqrt{\alpha}}8-1\Big)\ge l_*.
$$
In this case, the ODE in~\eqref{eq_sigma_0} implies
\begin{align*}
l'(r)&=l'\left(-\tfrac78\right)+\int_{-\frac78}^r\frac{E(l(\tau))}{\sigma_0(\tau)}\,d\tau 
\\
&\ge l_*-\sup_{\tau\in\mathbb R}|E(\tau)|\int_{-\frac78}^{\frac78}\frac1{\sigma_0(\tau)}\,d\tau \ge\frac{l_*}2>0,\qquad r\in\left(-\tfrac78,\tfrac78\right].
\end{align*}
Thus, the claim holds. Finally, assume $\alpha<\alpha_-$. Then, by~\eqref{sp_cu_scaling},
\begin{align*}
l\left(-\tfrac78\right)&=\tilde l\Big(\frac{\sqrt{\alpha}}8-1\Big)\le\tilde l\Big(\frac{\sqrt{\alpha_-}}8-1\Big)<-\tfrac34, \\
l'\left(-\tfrac78\right)&=\sqrt{\alpha}\,\tilde l'\Big(\frac{\sqrt{\alpha}}8-1\Big)\le\sqrt{\alpha_-}\,\sup\Big\{\tilde l'(r)\,\Big|\,r\in\Big[-1,\frac{\sqrt{\alpha_-}}8-1\Big]\Big\}<\tfrac14.
\end{align*}
This implies
\begin{align*}
l(r)&=l\left(-\tfrac78\right)+l'\left(-\tfrac78\right)\big(r+\tfrac78\big)+\int_{-\frac78}^r\int_{-\frac78}^\rho\frac{E(l(\tau))}{\sigma_0(\tau)}\,d\tau\,d\rho
\\
&<
-\tfrac14+48\int_{-\frac78}^{\frac78}\frac1{\sigma_0(\tau)}\,d\tau\le0,\qquad r\in\left(-\tfrac78,\tfrac78\right].
\end{align*}
Consequently, by the ODE in~\eqref{eq_sigma_0}, $l''>0$ on $\left(-\tfrac78,\tfrac78\right]$, which means $l'>0$ on this interval. We have shown that $\sigma_0$ possesses property~3.
\vspace{5pt}

\noindent\textit{Step 6.} It is easy to see that $l(r)$ remains negative for all $r$ if $\alpha\le0$. To prove property~5, assume that $l\left(\tfrac78\right)=0$ and $l'\left(\tfrac78\right)\ge\frac57$. Then $\alpha>0$ and, by property~3, $l'(r)>0$. Ergo, $l(r)\in[-1,0]$ for $r\in\left[-\tfrac78,\tfrac78\right]$. The ODE in~\eqref{eq_sigma_0} implies
\begin{align*}
l'(r)&=l'\left(\tfrac78\right)-\int_r^{\frac78}\frac{E(l(\tau))}{\sigma_0(\tau)}\,d\tau\ge \tfrac57-24\int_{-\frac78}^{\frac78}\frac1{\sigma_0(\tau)}\,d\tau>\tfrac47,\qquad r\in\left[-\tfrac78,\tfrac78\right].
\end{align*}
Consequently,
\begin{align*}
l\left(\tfrac78\right)\ge l\left(-\tfrac78\right)+1>0.
\end{align*}
This contradicts the assumption $l\big(\frac78\big)=0$. Thus, $\sigma_0$ must possess property~5.
\end{proof}
 
\subsection{Estimates and existence for the global problem}\label{so3existencesection}

Our plan is to use Brouwer degree theory to prove the existence of solutions to problem~\eqref{MEFL'}--\eqref{SCSO3} satisfying conditions~1 and~2 of Lemma~\ref{lem_so3_smth}. To that end, we introduce a continuous deformation of the function $\sigma_0$ into the function $\sigma$. More precisely, suppose
\begin{align}\label{zetap_def}
\zeta_p=(1-p)\sigma_0+p\sigma,\qquad p\in [0,1].
\end{align}
In order to be able to apply degree theory, we need the following result.

\begin{lemma}\label{MVA}
Let $l$ be a solution to equation~(\ref{MEFL}) on $(-1,1)$ with $\zeta_p$ given by~(\ref{zetap_def}), subject to the conditions
\begin{equation}\label{bound_ini_conds_so3}
l(-1)=-1, \qquad l(1)=1, \qquad l'(-1)=0=l'(1),\qquad l''(-1)=\alpha, \qquad l''(1)=-\beta.
\end{equation}
Then $\alpha<\alpha_{\max}$ and $\beta<\beta_{\max}$ for some positive $\alpha_{\max}$ and $\beta_{\max}$ independent of~$p$.
\end{lemma}

\begin{proof}
We will prove the existence of $\alpha_{\max}$. Similar arguments work for $\beta_{\max}$.
\vspace{5pt}

\noindent\textit{Step 1.}
 Assume that $\alpha_{\max}$ does not exist. If $r_1$ is the first point in $(-1,1)$ such that $l(r_1)=1$, one can easily show using~\eqref{MEFL} and~\eqref{bound_ini_conds_so3} that $l'>0$ on $[r_1,1)$. The equality $l(1)=1$ implies that $l\ge-1$ on~$[-1,1]$. Consequently, there exist sequences $(p_k)_{k=1}^\infty\subset[0,1]$ and $(\alpha_k)_{k=1}^\infty\subset(1,\infty)$ satisfying
\begin{align}\label{alpha_to_infty_so3}
\lim_{k\to\infty}\alpha_k=\infty,\qquad |l_k(r)|\le1,\qquad r\in[-1,0],
\end{align}
where $l_k$ is the solution to~\eqref{MEFL} with $p=p_k$, subject to
\begin{align*}
l_k(-1)=-1,\qquad l_k'(-1)=0,\qquad l_k''(-1)=\alpha_k.
\end{align*}
We will show that \eqref{alpha_to_infty_so3} gives rise to a contradiction by looking at appropriate re-scalings of~$l_k$. Choose $\tau_k\in[-1,0]$ so that
\begin{align*}
|l''_k(\tau_k)|+|l_k'(\tau_k)|^2=\sup_{t\in [-1,0]}\big(|l_k''(t)|+|l_k'(t)|^2\big)
\end{align*}
and denote
\begin{align*}
\lambda_k=\frac1{\sqrt{|l''(\tau_k)|+|l'(\tau_k)|^2}}.
\end{align*}
By passing to subsequences if necessary, we may assume that $(\lambda_k)_{k=1}^\infty$, $(\lambda_k^2\alpha_k)_{k=1}^{\infty}$ and $\big(\frac{\tau_k+1}{\lambda_k}\big)_{k=1}^\infty$ are all monotone. Formula~\eqref{alpha_to_infty_so3} implies that $(\lambda_k)_{k=1}^\infty$ goes to~0,
and the definitions of $\lambda_k$ and $\alpha_k$ imply that $0\le \lambda^2_k\alpha_k\le 1$. We will use the number $\lambda_k$ to re-scale~$l_k$.
\vspace{5pt}

\noindent\textit{Step 2.}
Assume that
\begin{align}\label{so3case1}
\lim_{k\to\infty}\frac{\tau_k+1}{\lambda_k}<\infty.
\end{align}
Define
\begin{align*}
\bar l_k(t)=l_k(\lambda_k(1+t)-1), \qquad t\in\Big[-1,\frac{1-\lambda_k}{\lambda_k}\Big].
\end{align*}
Then 
\begin{align*}
\bar l_k''(t)&=\frac{\lambda_k^2\big(\bar l_k(t)^3-\bar l_k(t)\big)}{\zeta_p(\lambda_k(1+t)-1)},\qquad \bar l_k''(-1)=\lambda_k^2\alpha_k.
\end{align*}
Formula~\eqref{zeta_p_asy} implies
\begin{align*}
\frac{\lambda_k^2}{\zeta_p(\lambda_k(1+t)-1)}&=\frac{1}{(1+t)^2}+\lambda_k^2S_{\zeta_p}(\lambda_k(1+t)-1).
\end{align*}
We conclude that
\begin{align*}
\bar l_k''(t)&=\frac{\bar l_k(t)^3-\bar l_k(t)}{(1+t)^2}+\lambda_k^2S_{\zeta_p}(\lambda_k(1+t)-1)\big(\bar l_k(t)^3-\bar l_k(t)\big).
\end{align*}
\vspace{5pt}

\noindent\textit{Step 3.}
Define
\begin{align*}
\tilde l_k(t)=\tilde l(\lambda_k\sqrt{\alpha_k}(1+t)-1),\qquad t\in\Big[-1,\frac{1-\lambda_k}{\lambda_k}\Big],
\end{align*}
where $\tilde l$ is the solution to~\eqref{eq_tildel_sp_cu}. Consider the function
\begin{align*}
s_k=\bar l_k-\tilde{l}_k.
\end{align*}
We will use Lemma~\ref{so3linearexistence} to show that the sequence $(s_k)_{k=1}^\infty$ converges to~0 on every bounded interval. This will contradict the inequality $|l_k(r)|\le1$ in~\eqref{alpha_to_infty_so3} if $\lim_{k\to \infty}\lambda_k\sqrt{\alpha_k}>0$. 
\vspace{5pt}

\noindent\textit{Step 4.}
Denote
\begin{align*}
T_0=\sup\{t\in(-1,\infty)\,|\,|\tilde l(t)|\le2\},\qquad \lambda=\sup_{t\in[-1,T_0]}|\tilde l'(t)|.
\end{align*}
Arguing as in the proof of Lemma~\ref{so3specialcurvature}, it is easy to prove that $T_0<\infty$. Without loss of generality, suppose that $\frac{1-\lambda_k}{\lambda_k}\ge T_0$ for all~$k$. Since $\lambda_k\sqrt{\alpha_k}(1+t)-1\le t$, we have
\begin{align*}
|\tilde l_k(t)|&=|\tilde l(\lambda_k\sqrt{\alpha_k}(1+t)-1)|\le2,
\\
|\tilde l_k'(t)|&=\lambda_k\sqrt{\alpha_k}|\tilde l'(\tau)||_{\tau=\lambda_k\sqrt{\alpha_k}(1+t)-1}\le\lambda, \qquad t\in[-1,T_0].
\end{align*}
The mean value theorem implies
\begin{align*}
|\bar l_k(t)+1|&\le(t+1)\sup_{\tau\in[-1,t]}|\bar l_k'(\tau)|\le\lambda_k(t+1)\sup_{\tau\in[-1,0]}|l_k'(\tau)|\le(t+1),
\\
|\tilde l_k(t)+1|&\le(t+1)\sup_{\tau\in[-1,t]}|\tilde l_k'(\tau)|\le\lambda(t+1),
\\
\big|\bar l_k(t)^3-\bar l_k(t)\big|&\le|\bar l_k(t)||\bar l_k(t)-1||\bar l_k(t)+1|\le2(t+1),\qquad t\in[-1,T_0].
\end{align*}
Next, we estimate
\begin{align*}
\bigg| s_k''(t)-\frac{2s_k(t)}{(1+t)^2}\bigg|
&\le\bigg|\frac{\bar l_k^3(t)-\bar l_k(t)}{(1+t)^2}-\frac{\tilde l_k^3(t)-\tilde l_k(t)}{(1+t)^2}-\frac{2(\bar l_k(t)-\tilde l_k(t))}{(1+t)^2}\bigg|
\\
&\hphantom{=}~+\lambda_k^2\big|S_{\zeta_p}(\lambda_k(1+t)-1)\big(\bar l_k(t)^3-\bar l_k(t)\big)\big|
\\
&\le\bigg|\frac{s_k(t)\big(\bar l_k^2(t)+\tilde l_k^2(t)+\bar l_k(t)\tilde l_k(t)-3\big)}{(1+t)^2}\bigg|+2S_*\lambda_k^2(1+t)
\\
&\le\bigg|\frac{s_k(t)\big((\bar l_k(t)+\tilde l_k(t)-1)(\bar l_k(t)+1)+(\tilde l_k(t)-2)(\tilde l_k(t)+1)\big)}{(1+t)^2}\bigg|+2S_*\lambda_k^2(1+t)
\\
&\le\bigg|\frac{4(1+\lambda)s_k(t)}{1+t}\bigg|+2S_*\lambda_k^2(1+t),\qquad t\in[-1,T_0],
\end{align*}
where
\begin{align*}
S_*=\sup\{|S_{\zeta_p}(t)|\,|\,t\in(-1,0],~p\in[0,1]\}<\infty.
\end{align*}
By Lemma~\ref{so3linearexistence},
\begin{align*}
\sup_{t\in[-1,-1+T_1]}\bigg|\frac{s_k(t)}{(1+t)^2}\bigg|\le\frac{T_1}4\sup_{t\in[-1,-1+T_1]}\bigg(\bigg|\frac{4(1+\lambda)s_k(t)}{(1+t)^2}\bigg|+2S_*\lambda_k^2\bigg)
\end{align*}
for all $T_1\le T_0+1$. Setting
\begin{align*}
T_1=\min\Big\{T_0+1,\frac1{2(1+\lambda)}\Big\},
\end{align*}
we conclude that
\begin{align*}
\sup_{t\in[-1,-1+T_1]}\bigg|\frac{s_k(t)}{(1+t)^2}\bigg|\le T_1S_*\lambda_k^2.
\end{align*}
Consequently, the sequence $(s_k)_{k=1}^{\infty}$ converges to 0 in the $C^2$-norm on $[-1,-1+T_1]$.
\vspace{5pt}

\noindent\textit{Step 5.}
Observe that
\begin{align*}
s_k''(t)=\psi_{1k}(t)s_k(t)+\lambda_k^2\psi_{2k}(t),
\end{align*}
where
\begin{align*}
\psi_{1k}(t)&=\frac{\bar l_k^2(t)+\tilde l_k^2(t)+\bar l_k(t)\tilde l_k(t)-1}{(1+t)^2},\\
\psi_{2k}(t)&=S_{\zeta_p}(\lambda_k(1+t)-1)\big(\bar l_k(t)^3-\bar l_k(t)\big).
\end{align*}
Given $T>T_1$, there exists $K\in \mathbb{N}$ such that the sequences $(s_k)_{k=K}^\infty$, $(s_k')_{k=K}^\infty$, $(s_k'')_{k=K}^\infty$,
$(\psi_{1k})_{k=K}^\infty$ and  $(\psi_{2k})_{k=K}^\infty$ are uniformly bounded and uniformly equicontinuous on~$[-1+T_1,-1+T]$. By the Arzel\`a--Ascoli theorem, passing to subsequences if necessary, we may assume that $(s_k)_{k=1}^\infty$ converges in the $C^2$-norm to some function $s$ while $(\psi_{1k})_{k=1}^\infty$ and $(\psi_{2k})_{k=1}^\infty$ converge in the sup-norm to some $\psi_1$ and $\psi_2$. Moreover, 
\begin{align*}
s''(t)=\psi_1(t)s(t),\qquad s(-1+T_1)=s'(-1+T_1)=0,
\end{align*}
which means $s$ is identically 0 on $[-1+T_1,-1+T]$. We conclude that $(s_k)_{k=1}^\infty$ converges to 0 in the $C^2$-norm on $[-1,-1+T]$.
\vspace{5pt}

\noindent\textit{Step 6.}
Choose $T$ so that 
\begin{align*}
T>\sup_{k\in\mathbb N}\frac{\tau_k+1}{\lambda_k}.
\end{align*}
Then
\begin{align}\label{deep_est_contra_so3}
\sup_{[-1,-1+T]}\big(|\bar l_k''(t)|+|\bar l_k'(t)|^2\big)=\lambda_k^2\big(|l''_k(\tau_k)|+|l_k'(\tau_k)|^2\big)=1
\end{align}
for all $k$. If
\begin{align*}
\lim_{k\to\infty}\lambda_k\sqrt{\alpha_k}=0,
\end{align*}
then $\tilde l_k$ converges to $-1$ in the $C^2$-norm on $[-1,-1+T]$, which means
\begin{align*}
\bar l_k(t)=s_k(t)+\tilde l_k(t)
\end{align*}
also converges to~$-1$. This contradicts~\eqref{deep_est_contra_so3}. If
\begin{align*}
\lim_{k\to\infty}\lambda_k\sqrt{\alpha_k}>0,
\end{align*}
choose $T$ so that 
\begin{align*}
\tilde l_k(-1+T)=\tilde l(\lambda_k\sqrt{\alpha_k}T-1)>2
\end{align*}
for all $k$. Then
\begin{align*}
\bar l_k(-1+T)=s_k(-1+T)+\tilde l_k(-1+T)>1,
\end{align*}
provided $k$ is large. This is impossible in light of~\eqref{alpha_to_infty_so3}. The contradictions we obtained prove the assertion of the lemma in the case where~\eqref{so3case1} holds.
\vspace{5pt}

\noindent\textit{Step 7.}
Assume that
\begin{align*}
\lim_{k\to\infty}\frac{\tau_k+1}{\lambda_k}=\infty.
\end{align*}
For $k\in\mathbb N$, consider the function
\begin{align*}
\hat l_k(t)=l_k(\tau_k+\lambda_kt),\qquad t\in\Big[-\frac{\tau_k+1}{\lambda_k},-\frac{\tau_k}{\lambda_k}\Big].
\end{align*}
Clearly,
\begin{align*}
|\hat l_k''(t)|&\le\frac{\lambda_k^2|\hat l_k(t)^3-\hat l_k(t)|}{(1+\tau_k+\lambda_kt)^2}+2\lambda_k^2S_*\le\frac{2\lambda_k^2}{(1+\tau_k+\lambda_kt)^2}+2\lambda_k^2S_*
\end{align*}
on $\big(-\frac{\tau_k+1}{\lambda_k},-\frac{\tau_k}{\lambda_k}\big]$. This shows that $\hat l_k''$ converges uniformly to 0 on $[-T,0]$ for every $T>0$. Also,
\begin{align*}
|\hat l_k''(0)|+|\hat l_k'(0)|^2=\lambda_k^2\big(|l_k''(\tau_k)|+|l_k'(\tau_k)|^2\big)=1,
\end{align*}
which means $\hat l_k'(0)$ goes to 1 as $k\to\infty$. Consequently,
\begin{align*}
\lim_{k\to\infty}\sup_{t\in[-T,0]}|\hat l_k'(t)-1|=\lim_{k\to\infty}\sup_{t\in[-T,0]}|\hat l_k'(t)-\hat l_k'(0)|\le T\lim_{k\to\infty}\sup_{t\in[-T,0]}|\hat l_k''(t)|=0.
\end{align*}
We conclude that
\begin{align*}
|\hat l_k(0)-\hat l_k(-T)|\ge T\inf_{t\in[-T,0]}|\hat l_k'(t)|>2
\end{align*}
if $T=3$ and $k$ is sufficiently large. At the same time,
\begin{align*}
|\hat l_k(0)-\hat l_k(-T)|\le |\hat l_k(0)|+|\hat l_k(-T)|\le2
\end{align*}
by~\eqref{alpha_to_infty_so3}. This contradiction completes the proof.
\end{proof}

We are now in a position to take the final step in the proof of Theorem~\ref{so3exist}. Lemma~\ref{LocalExistence} and the Lipschitz continuity of $E$ on $\mathbb R$ imply that~\eqref{MEFL} possesses a unique solution on $(-1,0]$ such that conditions~\eqref{ICL''-1} hold and a unique solution on $[0,1)$ such that~\eqref{ICL''1} hold. We denote these solutions $l^{p,\alpha}_{-}$ and  $l^{p,\beta}_{+}$, respectively. Consider the function $J_p:\mathbb{R}^2\to \mathbb{R}^2$ given by 
\begin{align*}
J_p(\alpha,\beta)=\big((l^{p,\beta}_{+}-l^{p,\alpha}_{-})(0),(l^{p,\beta}_{+}-l^{p,\alpha}_{-})'(0)\big).
\end{align*}
Our proof of Theorem~\ref{so3exist} will rely on the analysis of the degree-theoretic properties of~$J_p$.

\begin{lemma}\label{MET}
There exists a function $l:[-1,1]\to\mathbb R$ that solves problem~(\ref{MEFL'})--(\ref{SCSO3}) and satisfies conditions~1 and~2 of Lemma~\ref{lem_so3_smth}.
\end{lemma}

\begin{proof}
Denote
\begin{align*}
\Omega=[0,\alpha_{\max}]\times[0,\beta_{\max}],
\end{align*}
where $\alpha_{\max}$ and $\beta_{\max}$ are given by Lemma~\ref{MVA}. For every $p\in[0,1]$, the function $J_p$ is continuous on~$\Omega$.  It has no zeros in $\partial\Omega$ by Lemma~\ref{MVA}. Let us find the winding number of $J_0$ around the origin as we follow~$\partial\Omega$ anticlockwise starting at the origin; see Figure~1.

\begin{figure}
\centering
\includegraphics[width=140mm]{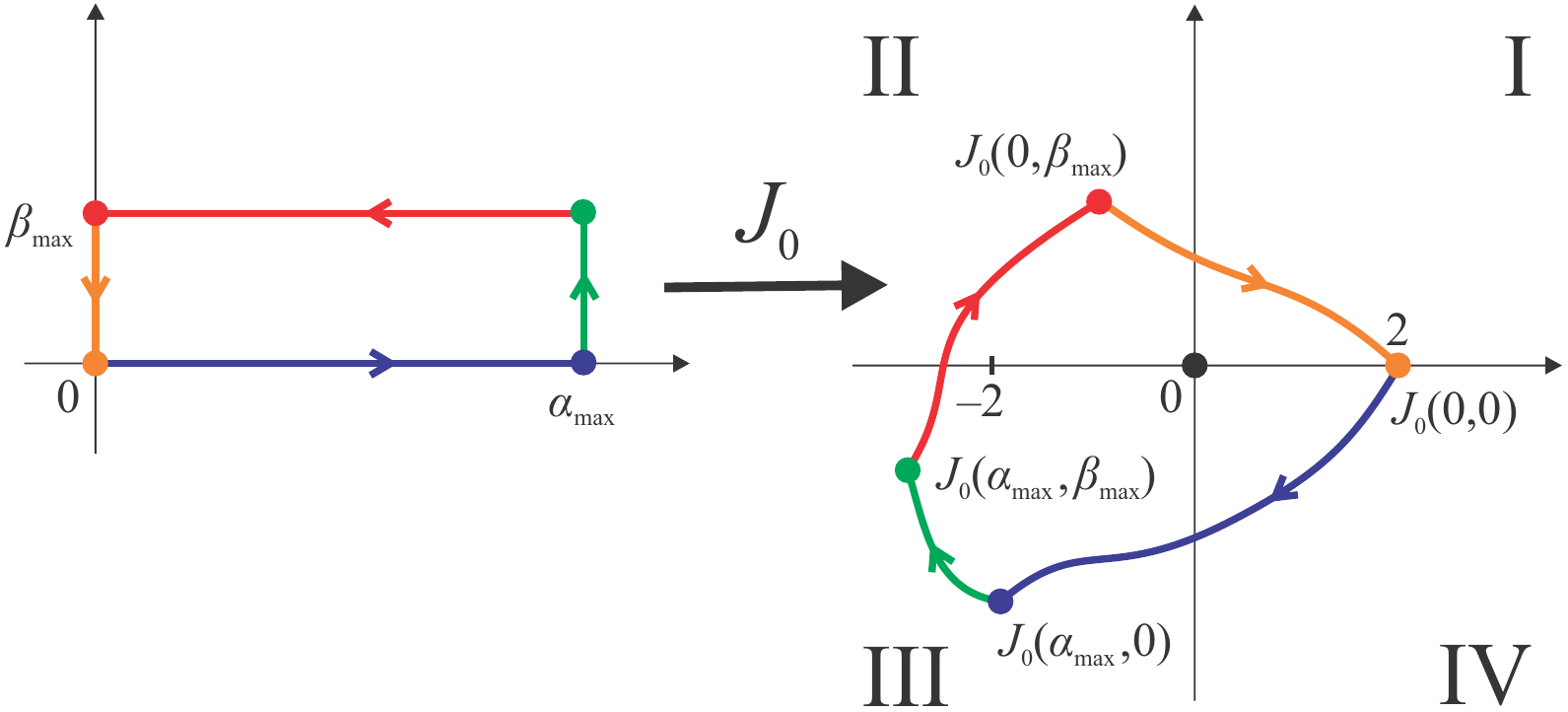}
\caption{The curve $J_0(\partial\Omega)$}
\end{figure}

Since $l^{0,0}_+$ and $l^{0,0}_-$ are constant, $J_{0}(0,0)=(2,0)$. Property~3 from Lemma~\ref{so3specialcurvature} implies that $J_0((0,\alpha_{\max}]\times\{0\})$ lies below the horizontal axis. Without loss of generality, assume that $\min\{\alpha_{\max},\beta_{\max}\}>\alpha_0$. By Property~4 from Lemma~\ref{so3specialcurvature}, this means that $J_0(\alpha_{\max},0)$ is in the third quadrant. Continuing to argue in this way, we can easily show that 
\begin{align*}
J_0((\{\alpha_{\max}\}\times[0,\beta_{\max}])\cup([0,\alpha_{\max}]\times\{\beta_{\max}\}))\qquad\mbox{and}\qquad J_0(\{0\}\times[0,\beta_{\max}))
\end{align*}
lie to the left of the vertical axis and above the horizontal axis, respectively. This enables us to conclude that the winding number of $J_0(\partial\Omega)$ around the origin is~$-1$. Thus, $\deg(J_0,\Omega)\neq 0$. 

The Leray--Schauder continuation theorem (see, e.g.,~\cite[Theorem~2.2]{SDT}) implies the existence of a compact connected set $\Omega_*\subset[0,1]\times\Omega$ such that $J_{p}(\alpha,\beta)=0$ for all $(p,\alpha,\beta)\in\Omega_*$ and the intersections $\Omega_*\cap(\{0\}\times\Omega)$ and $\Omega_*\cap(\{1\}\times\Omega)$ are nonempty. Given $(p,\alpha,\beta)\in[0,1]\times\Omega$, consider the function
\begin{align*}
l^{p,\alpha,\beta}=
\begin{cases}
l^{p,\alpha}_-&\mbox{on}~[-1,0], \\ 
l^{p,\beta}_+&\mbox{on}~(0,1].
\end{cases}
\end{align*}
If $(p,\alpha,\beta)\in\Omega_*$, this function solves~\eqref{MEFL} on $(-1,1)$ subject to the conditions
\begin{align}\label{bc_both_sides_degree}
l^{p,\alpha,\beta}(-1)=-1,\qquad (l^{p,\alpha,\beta})'(-1)=0,\qquad l^{p,\alpha,\beta}(1)=1,\qquad (l^{p,\alpha,\beta})'(1)=0.
\end{align}
Choose $(\alpha_*,\beta_*)\in\Omega$ so that $(1,\alpha_*,\beta_*)\in\Omega_*$. We claim that the assertion of the lemma holds with~$l=l^{1,\alpha_*,\beta_*}$. It remains to verify that $l^{1,\alpha_*,\beta_*}$ satisfies conditions~1 and~2 of Lemma~\ref{lem_so3_smth}.

Arguing as in first two steps of the proof of Theorem~\ref{so2existence}, we can extend $l^{1,\alpha_*,\beta_*}$ to a smooth function on $(-3,3)$ that is even about $r=-1$ and about $r=1$. Thus, $l^{1,\alpha_*,\beta_*}$ satisfies condition~1. Let us show that $(l^{1,\alpha_*,\beta_*})'$ is positive. Denote
$$
\Omega_+=\{(p,\alpha,\beta)\in\Omega_*\,|\,(l^{p,\alpha,\beta})'>0~\mbox{on}~(-1,1)\}.
$$
Property~3 from Lemma~\ref{so3specialcurvature} implies that $\Omega_+\ne\emptyset$. We will prove that $\Omega_+$ is both open and closed in $\Omega_*$. The inequality $(l^{1,\alpha_*,\beta_*})'>0$ will follow immediately.

Consider a sequence $((p_n,\alpha_n,\beta_n))_{n=1}^\infty\in\Omega_*\setminus\Omega_+$. Then $(l^{p_n,\alpha_n,\beta_n})'(r_n)\le0$ for some  $(r_n)_{n=1}^\infty\subset(-1,1)$. Assume that
$$
\lim_{n\to\infty}(p_n,\alpha_n,\beta_n)=(p,\alpha,\beta)\in\Omega_*,\qquad \lim_{n\to\infty}r_n=r_0\in[-1,1].
$$
Using the ODE~\eqref{MEFL} and the boundary conditions~\eqref{bc_both_sides_degree}, one can easily show that either $r_0\in(-1,1)$ and $(l^{p,\alpha,\beta})'(r_0)\le0$ or $r_0=(-1)^i$ and $(-1)^i(l^{p,\alpha,\beta})''(r_0)\ge0$ for some $i\in\{0,1\}$. In both cases, $l^{p,\alpha,\beta}$ cannot lie in~$\Omega_+$. This means $\Omega_+$ is open in $\Omega_*$.

Assume that $(p,\alpha,\beta)$ lies in the boundary of $\Omega_+$ but not in $\Omega_+$ itself. In this case, $(l^{p,\alpha,\beta})'(r_0)=0$ for some $r_0\in(-1,1)$. If $l^{p,\alpha,\beta}(r_0)$ equals $-1$, 0 or 1, then $l^{p,\alpha,\beta}$ must be constant by uniqueness of solutions to ODEs, which is impossible. If $l^{p,\alpha,\beta}(r_0)$ takes some other value, then $r_0$ is a strict extremum point for $l^{p,\alpha,\beta}$ by the second derivative test. However, this is again impossible because $(l^{p,\alpha,\beta})'\ge0$ on~$(-1,1)$. Thus, $\Omega_+$ is closed in $\Omega_*$.
\end{proof}

We now obtain Theorem~\ref{so3exist} by combining Lemmas~\ref{lem_so3_smth} and~\ref{MET}.

\section{Non-Uniqueness}\label{sec_nonuniq}

Consider a cohomogeneity one action of $SO(3)$ on~$\mathbb S^3$ as in Section~\ref{SectionexistenceSO3}. Given a symmetric positive-definite $SO(3)$-invariant tensor field $Y$, we fix a $Y$-geodesic between the two singular orbits and identify the principal part of $\mathbb S^3$ with the product~$(-1,1)\times \mathbb{S}^2$. The reflection of $(-1,1)$ about the middle induces an action of $\mathbb Z_2$ on~$\mathbb S^3$, which we call the $\mathbb Z_2$-action associated with~$Y$. Let us state the main result of this section.

\begin{theorem}\label{thm_nonuniq}
There exists a symmetric positive-definite $SO(3)$-invariant tensor field $Y$ on~$\mathbb S^3$ such that the equation $X(g)=Y$ is satisfied by at least three distinct $SO(3)$-invariant Riemannian metrics~$g$, one of which (and only one) is invariant under the $\mathbb Z_2$-action associated with~$Y$.
\end{theorem}
This result refutes R.~Hamilton's conjecture on the uniqueness of solutions to the prescribed cross curvature problem; see~\cite[Chapter 2]{IG08}. In the course of the proof, we will construct three metrics satisfying $X(g)=Y$. It will be clear from our arguments that two of them (the ones that are not $\mathbb Z_2$-invariant) are isometric. We discuss this in greater detail at the end of the section.

Lemma~\ref{lem_so3_smth} reduces the proof of Theorem~\ref{thm_nonuniq} to the analysis of the ODE~\eqref{MEFL'}. More precisely, it suffices to establish the following.

\begin{lemma}\label{lem_nonuniq_ODE}
There exist a function $\sigma$ and distinct functions $l_-$, $l_+$ and $l_0$ on $[-1,1]$ such that the following statements hold:
\begin{enumerate}
\item
The equality $\sigma(r)=(r+1)^2$ is satisfied for $r\in\left[-1,-\frac78\right]$, and $\sigma$ is even about $r=0$.

\item
Each of the functions $l_-$, $l_+$ and $l_0$ solves the boundary-value problem~(\ref{MEFL'})--(\ref{SCSO3}) and possesses properties~1 and~2 from Lemma~\ref{lem_so3_smth}.

\item
The function $l_0$ is odd about $r=0$, while $l_-$ and $l_+$ are not.
\end{enumerate}
\end{lemma}

Let us prove Lemma~\ref{lem_nonuniq_ODE}. In line with statement~1, set $\sigma(r)=(r+1)^2$ for $r\in\left[-1,-\frac78\right]$. Denote by $l_{\rm{sm}}:\big[\!-1,-\tfrac{7}{8}\big]\to \mathbb{R}$ the solution to equation~\eqref{MEFL'} satisfying
\begin{align*}
l_{\rm{sm}}(-1)=-1,\qquad l_{\rm{sm}}\left(-\tfrac{7}{8}\right)=0,\qquad l_{\rm{sm}}'(-1)=0.
\end{align*}
One finds such a solution by re-scaling the function $\tilde l$ introduced in Section~\ref{SectionexistenceSO3}. Clearly, $l_{\rm{sm}}'>0$ on $(-1,-\tfrac{7}{8}\big]$. The ODE for $l_{\rm{sm}}$ implies that $l_{\rm{sm}}''\ge0$, which means the derivative 
$$l^*=l_{\rm{sm}}'\left(-\tfrac{7}{8}\right)$$
is greater than or equal to~8. Our plan is to construct $l_-$ by extending $l_{\rm{sm}}$ to all of $[-1,1]$ in a specific way. To do so, we resort to degree theory again.

Let $\sigma_1$ be a smooth positive even function on $(-1,1)$ equal to $(1+r)^2$ on $\big(-1,-\frac{7}{8}\big]$. For $p\in[0,1]$, define
\begin{align*}
\sigma_p=(1-p)\sigma_0+p\sigma_1,
\end{align*}
where $\sigma_0$ is constructed in Section~\ref{subsec_sp_curv}. 
Lemma~\ref{LocalExistence} and the global Lipschitz continuity of $E$ imply the existence of a unique solution on the interval $\big[-1,\tfrac78\big]$ to the initial-value problem
\begin{align}\label{equationsso3nonuniqueness}
l''&=\frac{E(l)}{\sigma_p},\qquad
l(-1)=-1, \qquad l'(-1)=0, \qquad l''(-1)=\alpha.
\end{align}
We denote this solution $l^{p,\alpha}$ and introduce the map
\begin{align*}
\mathbb R\ni\alpha\to K_p(\alpha)=l^{p,\alpha}\big(\tfrac78\big)\in\mathbb R.
\end{align*}
Lemma~\ref{so3specialcurvature} (specifically, property~4) enables us to choose $\alpha_{\rm{mx}}>0$ (independent of $p$) so that $l^{p,\alpha}\big(\tfrac78\big)>l^{p,\alpha}\big(-\tfrac78\big)>1$ when $\alpha\ge\alpha_{\rm{mx}}$. To produce an extension of the function $l_{\rm{sm}}$, we will exploit some degree-theoretic properties of~$K_p$ in~$[0,\alpha_{\rm{mx}}]$.

\begin{lemma}\label{lem_deg_nonuniq}
There exists a compact connected set $\Theta\subset[0,1]\times[0,\alpha_{\rm{mx}}]$ such that $K_{p}(\alpha)=0$ for all $(p,\alpha)\in\Theta$ and the intersections $\Theta\cap(\{0\}\times[0,\alpha_{\rm{mx}}])$ and $\Theta\cap(\{1\}\times[0,\alpha_{\rm{mx}}])$ are nonempty.
\end{lemma}

\begin{proof}
The assertion of the lemma follows form the Leray--Schauder continuation theorem (see, e.g.,~\cite[Theorem 2.2]{SDT}) applied to $K_p$ on $[0,\alpha_{\rm{mx}}]$. Indeed, our choice of $\alpha_{\rm{mx}}$ ensures that $K_p$ does not have any zeros on the boundary of this interval. It remains to verify that $\deg(K_0,[0,\alpha_{\rm{mx}}])\ne0$. Since
$$
K_0(0)=-1<0\qquad\mbox{and}\qquad K_0(\alpha_{\rm{mx}})=l^{0,\alpha_{\rm{mx}}}\big(\tfrac78\big)>1>0,
$$
we have $\deg(K_0,[0,\alpha_{\rm{mx}}])=1$.
\end{proof}

As in the proof of Lemma~\ref{MET}, one can easily show that, whenever $(p,\alpha)\in\Theta$, the function $l^{p,\alpha}$ has positive derivative on $\left(-1,\tfrac78\right]$ and extends to a smooth function on $\left[-\frac{23}8,\tfrac78\right]$ that is even about $r=-1$.

\begin{lemma}\label{lem_choose_sigma1}
It is possible to choose the function $\sigma_1$ above so that
$$
(l^{1,\alpha})'\big(\tfrac78\big)>l_{\rm{sm}}'\left(-\tfrac{7}{8}\right)=l^*
$$
for all $\alpha$ satisfying $(1,\alpha)\in\Theta$.
\end{lemma}

\begin{proof}
Set
\begin{align*}
\delta=\min\Big\{\frac1{2(l^*+1)},\frac7{16}\Big\}.
\end{align*}
Fix $\alpha$ such that $(1,\alpha)\in\Theta$. Since $l^{1,\alpha}$ has positive derivative on $\left(-1,\tfrac78\right)$ and $l^{1,\alpha}\big(\tfrac78\big)=K_{1}(\alpha)=0$, the values of $l^{1,\alpha}$ on $\left(-1,\tfrac78\right)$ lie between $-1$ and~$0$. The ODE in~\eqref{equationsso3nonuniqueness} implies that $l^{1,\alpha}$ is convex on this interval. We consider three cases.

Assume that $l^{1,\alpha}\big(\tfrac78-\delta\big)>-\frac{8\delta}{15}$. Then 
$$
(l^{1,\alpha})'(r)>(l^{1,\alpha})'\big(\tfrac78-\delta\big)>\tfrac8{15}
$$
for $r\in\big(\tfrac78-\delta,\tfrac78\big)$ by convexity. However, this means $l^{1,\alpha}\big(\tfrac78\big)>0$, which is false.

Assume that 
\begin{align}\label{case2so3nonunique}
    l^{1,\alpha}\big(\tfrac78-2\delta\big)\ge-1+\delta,\qquad l^{1,\alpha}\big(\tfrac78-\delta\big)\le-\tfrac{8\delta}{15}.
\end{align}
On the interval $\big[\tfrac78-2\delta,\tfrac78-\delta\big]$, we have
$$
(l^{1,\alpha})^3-l^{1,\alpha}>\delta_0
$$
for some $\delta_0>0$ depending only on~$\delta$. Choose $\sigma_1$ so that, on this interval, $\sigma_1<\frac{\delta^2\delta_0}{2}$. By the ODE in~\eqref{equationsso3nonuniqueness}, this means $(l^{1,\alpha})''\ge\frac{2}{\delta^2}$, which contradicts 
\eqref{case2so3nonunique}. 

Finally, assume that $l^{1,\alpha}\big(\tfrac78-2\delta\big)<-1+\delta$. Using convexity again, we find
$$
(l^{1,\alpha})'\big(\tfrac78\big)>\frac{l^{1,\alpha}\big(\tfrac78\big)-l^{1,\alpha}\big(\tfrac78-2\delta\big)}{2\delta}>\frac{1-\delta}{2\delta}\ge l^*.
$$
\end{proof}

Choose $\sigma_1$ as in Lemma~\ref{lem_choose_sigma1}. By property~5 from Lemma~\ref{so3specialcurvature},
$$
(l^{0,\alpha_0})'\big(\tfrac78\big)<\tfrac57<8\le l^*<(l^{1,\alpha_1})'\big(\tfrac78\big)
$$
whenever $\alpha_0$ and $\alpha_1$ satisfy $(0,\alpha_0)\in\Theta$ and $(1,\alpha_1)\in\Theta$. The compactness and the connectedness of $\Theta$ imply the existence of $(p_*,\alpha_*)\in\Theta$ such that $(l^{p_*,\alpha_*})'\big(\tfrac78\big)=l^*$. We set $\sigma=\sigma_{p_*}$. Using the notation $l_{\rm{lg}}=l^{p_*,\alpha_*}$, define
\begin{align*}
l_-(r)=
\begin{cases}
l_{\rm{sm}}(r) &\mbox{for}~r\in\big[\!-1,-\frac78\big], \\ 
-l_{\rm{lg}}(-r) &\mbox{for}~r\in\big(\!-\frac78,1\big],
\end{cases}
\qquad
l_+(r)=
\begin{cases}
l_{\rm{lg}}(r) &\mbox{for}~r\in\big[\!-1,\frac78\big], \\ 
-l_{\rm{sm}}(-r) &\mbox{for}~r\in\big(\frac78,1\big].
\end{cases}
\end{align*}
Clearly, these functions solve the boundary-value problem~\eqref{MEFL'}--\eqref{SCSO3} and have positive derivatives on $(-1,1)$. They are distinct, and neither of them is odd. Arguing as in the first two steps of the proof of Theorem~\ref{so2existence}, one easily shows that they possess property~1 from Lemma~\ref{lem_so3_smth}. Now that $\sigma$, $l_-$ and $l_+$ are at hand, it remains to construct~$l_0$.

\begin{proof}[Proof of Lemma~\ref{lem_nonuniq_ODE}]
Denote
\begin{align*}
\alpha^*=\sup\big\{\alpha\ge0\,\big|\,-1\le l^{p_*,\alpha}\le0~\mbox{and}~\big(l^{p_*,\alpha}\big)'\ge0~\mbox{on}~(-1,0)\big\}.
\end{align*}
One can easily check that $0<\alpha^*<\infty$ using Lemmas~\ref{so3specialcurvature} and the fact that $\alpha_*$ lies in the set on the right-hand side. Define
\begin{align*}
l_0(r)=
\begin{cases}
l^{p_*,\alpha^*}(r) &\mbox{for}~r\in\big[\!-1,0\big], \\ 
-l^{p_*,\alpha^*}(-r) &\mbox{for}~r\in\big(0,1\big].
\end{cases}
\end{align*}
Recall that $l^{p_*,\alpha}$ and $(l^{p_*,\alpha})'$ depend continuously on~$\alpha$. As a consequence, $-1\le l_0\le0$ and $l_0'\ge0$ on~$(-1,0)$. Moreover, $l_0''$ is nonnegative on~$(-1,0)$. Therefore, $l_0'$ is greater than~0 on~$(-1,0)$, as is the one-sided derivative of $l_0$ at $r=0$. The definition of $\alpha^*$ implies that $l_0(r)=0$ for some $r\in[-1,0]$. This is only possible if~$r=0$. We conclude that $l_0$ is continuous on $[-1,1]$ and has positive derivative on~$(-1,1)$. Moreover, this function solves~\eqref{MEFL'}--\eqref{SCSO3}. Arguing as in the proof of Theorem~\ref{so2existence}, one easily shows that it possesses property~1 from Lemma~\ref{lem_so3_smth}.
\end{proof}

Denote by $g_-$, $g_+$ and $g_0$ the metrics on $\mathbb S^3$ associated with $l_-$, $l_+$ and~$l_0$. Clearly, all three of them have cross curvature~$Y$. Moreover, $g_0$ is invariant under the $\mathbb Z_2$-action given by~$Y$, while $g_-$ and $g_+$ are not. This proves Theorem~\ref{thm_nonuniq}.

The reflection of the interval $(-1,1)$ about its centre induces an isometry between $g_-$ and~$g_+$. However, the following result holds.

\begin{theorem}\label{thm_isometry}
The metrics $g_-$ and $g_+$ are not isometric to~$g_0$.
\end{theorem}

\begin{proof}
For the sake of contradiction, assume that there exists a diffeomorphism $\iota:\mathbb S^3\to\mathbb S^3$ such that $g_+=\iota^*g_0$. At the beginning of Section~\ref{sec_nonuniq}, we fixed a cohomogeneity one action of $SO(3)$ on~$\mathbb S^3$. Denoting it by $\mathcal A$, we introduce another action $\tilde{\mathcal A}$ of $SO(3)$ on~$\mathbb S^3$ via the formula
$$
\tilde{\mathcal A}(\kappa)=\iota^{-1}\circ\mathcal A(\kappa)\circ\iota,\qquad \kappa\in SO(3).
$$
Clearly, $g_+$ must be invariant with respect to both $\mathcal A$ and~$\tilde{\mathcal A}$. If the $\mathcal A$-orbit of one point in~$\mathbb S^3$ intersects the $\tilde{\mathcal A}$-orbit of another point, then the scalar curvature of $g_+$ must be the same at these two points. We will produce a counterexample to this statement. The resulting contradiction will prove Theorem~\ref{thm_isometry}.

The action $\mathcal A$ and our choice of a $Y$-geodesic $\gamma$ on $\mathbb S^3$ yield an identification between the principal part of $\mathbb S^3$ and the product $(-1,1)\times\mathbb S^2$. Let $s_{\rm{n}}=\gamma(-1)$ and $s_{\rm{s}}=\gamma(1)$ be the singular orbits of~$\mathcal A$. If an $\mathcal A$-invariant metric $g$ on $\mathbb S^3$ satisfies~\eqref{so3metric}, we can compute its scalar curvature as in the proof of Lemma~\ref{lem_so3_smth}. We find
\begin{align*}
S(g)=\frac{4\sigma}{1-l^2}-\frac{2(l^2-1)}{(l')^2},
\end{align*}
where $l$ and $\sigma$ are defined by~\eqref{lsigmaso3}. L'H\^opital's rule implies
\begin{align*}
S(g)(s_{\rm{n}})=\frac{6}{l''(-1)},\qquad S(g)(s_{\rm{s}})=-\frac{6}{l''(1)}. 
\end{align*}

The points
$$
\tilde s_{\rm{n}}=\iota^{-1}(s_{\rm{n}})\qquad \mbox{and}\qquad \tilde s_{\rm{s}}=\iota^{-1}(s_{\rm{s}})
$$
are the singular orbits of~$\tilde{\mathcal A}$. Since scalar curvature is preserved by isometries,
$$
S(g_+)(\tilde s_{\rm{n}})=S(g_0)(s_{\rm{n}})=\frac{6}{l_0''(-1)},\qquad S(g_+)(\tilde s_{\rm{s}})=S(g_0)(s_{\rm{s}})=-\frac{6}{l_0''(1)}.
$$
Lemma~\ref{LocalExistence} and the fact that $l_+$ and $l_0$ are distinct show that
\begin{align*}
S(g_+)(\tilde s_{\rm{n}})=\frac{6}{l_0''(-1)}&\ne\frac{6}{l_+''(-1)}=S(g_+)(s_{\rm{n}}), \\
S(g_+)(\tilde s_{\rm{s}})=-\frac{6}{l_0''(1)}&\ne\frac{6}{l_+''(-1)}=S(g_+)(s_{\rm{n}}).
\end{align*}
Analogously, $S(g_+)(\tilde s_{\rm{n}})\ne S(g_+)(s_{\rm{s}})$ and $S(g_+)(\tilde s_{\rm{s}})\ne S(g_+)(s_{\rm{s}})$. The points $\tilde s_{\rm{n}}$, $\tilde s_{\rm{s}}$, $s_{\rm{n}}$ and $s_{\rm{s}}$ are pairwise distinct. We will show that the $\mathcal A$-orbit of $\tilde s_{\rm{n}}$ or $\tilde s_{\rm{s}}$ intersects the $\tilde{\mathcal A}$-orbit of~$s_{\rm{n}}$. Consequently, $S(g_+)(\tilde s_{\rm{n}})$ or $S(g_+)(\tilde s_{\rm{s}})$ must equal~$S(g_+)(s_{\rm{n}})$. This contradiction will complete the proof.

Using the action $\mathcal A$, we identified the principal part of $\mathbb S^3$ with the product $(-1,1)\times\mathbb S^2$. Let
$$
I:(-1,1)\times\mathbb S^2\to\mathbb S^3\setminus\{s_{\rm{s}},s_{\rm{n}}\}
$$
be the diffeomorphism that realises this identification. Then $\tilde s_{\rm{n}}=I(r_{\rm{n}},\omega_{\rm{n}})$ and $\tilde s_{\rm{s}}=I(r_{\rm{s}},\omega_{\rm{s}})$ for some $r_{\rm{n}},r_{\rm{s}}\in(-1,1)$. The $\mathcal A$-orbits of $\tilde s_{\rm{n}}$ and $\tilde s_{\rm{s}}$ are the sets
\begin{align}\label{two_orbits}
\{(r_{\rm{n}},\omega)\,|\,\omega\in\mathbb S^2\}\qquad\mbox{and}\qquad \{(r_{\rm{s}},\omega)\,|\,\omega\in\mathbb S^2\}.
\end{align}
Without loss of generality, assume that $r_{\rm{n}}\le r_{\rm{s}}$. Fix a continuous curve $\omega:[0,1]\to\mathbb S^2$ such that $\omega(0)=\omega_{\rm{n}}$ and $\omega(1)=\omega_{\rm{s}}$. Denote by $\mathcal O$ the $\tilde{\mathcal A}$-orbit of~$s_{\rm{n}}$. We will show that the first set in~\eqref{two_orbits} intersects~$\mathcal O$. This will complete the proof.

Just like $\mathcal A$, the action $\tilde{\mathcal A}$ yields an identification between the principal part of the sphere $\mathbb S^3$ and the product $(-1,1)\times\mathbb S^2$. Using this identification, we can split $\mathbb S^3$ into a union of connected disjoint sets
\begin{align*}
\mathbb S^3=\mathcal S_{\rm{n}}\cup\mathcal O\cup\mathcal S_{\rm{s}}
\end{align*}
with $\tilde s_{\rm{n}}\in\mathcal S_{\rm{n}}$ and $\tilde s_{\rm{s}}\in\mathcal S_{\rm{s}}$. Both $\mathcal S_{\rm{n}}$ and $\mathcal S_{\rm{s}}$ are open. The orbit $\mathcal O$ is their common boundary. The curve $\gamma:[0,1]\to\mathbb S^3$ given by
\begin{align*}
\gamma(t)=I((1-t)r_{\rm{n}}+tr_{\rm{s}},\omega(t))
\end{align*}
connects the points $\tilde s_{\rm{n}}\in\mathcal S_{\rm{n}}$ and~$\tilde s_{\rm{s}}\notin\mathcal S_{\rm{n}}$. Consequently, it must have a nonempty intersection with~$\partial\mathcal S_{\rm{n}}$. We conclude that there exists a point $(r',\omega')\in(-1,1)\times\mathbb S^2$ such that $r'\ge r_{\rm{n}}$ and $I(r',\omega')\in\mathcal O$. Since $\mathcal O$ is connected and $s_{\rm{n}}$ also lies in $\mathcal O$, the inclusion $I(r_{\rm{n}},\omega'')\in\mathcal O$ holds for some~$\omega''$. Thus, the first set in~\eqref{two_orbits} intersects~$\mathcal O$.
\end{proof}

\section{Dynamics for left-invariant metrics on $SU(2)$}\label{SU2}

Fix a bi-invariant background metric $g_0$ on~$SU(2)$ satisfying
\begin{align*}
g_0(X,Y)=-2\tr(XY),\qquad X,Y\in\mathfrak{su}(2).
\end{align*}
Choose a $g_0$-orthonormal basis $\{e_1,e_2,e_3\}$ of $\mathfrak{su}(2)$. Then $[e_i,e_j]$ must equal $e_k$ up to sign for every permutation $(i,j,k)$ of $\{1,2,3\}$. Replacing $e_1$ by $-e_1$ if necessary, we may assume that
$[e_i,e_j]=e_k$
as long as $(i,j,k)$ is even. Given a left-invariant metric~$g$ on~$SU(2)$, define
\begin{align*}
\lambda(g)&=\sup\Big\{\frac{g(X,X)}{g(Y,Y)}\,\Big|\,X,Y\in\mathfrak{su}(2)~\mbox{and}~g_0(X,X)=g_0(Y,Y)=1\Big\}.
\end{align*}
This is the ratio of the largest eigenvalue of the matrix of $g$ in $\{e_1,e_2,e_3\}$ to the smallest.
Our next theorem reveals that $\lambda(g)$ increases as one passes from $g$ to~$X(g)$. Several results of this nature are known for the Ricci curvature; see, e.g.,~\cite{PR19,BPRZ21}.

\begin{theorem}\label{thm_SU2}
For every left-invariant positive-definite symmetric (0,2)-tensor field~$Y$ on~$SU(2)$, there exists a unique left-invariant metric~$g$ such that ${X(g)=Y}$. Moreover,
\begin{align}\label{mon_qtty}
\lambda(g)\le\lambda(Y)
\end{align}
with equality holding if and only if $g=\rho g_0$ for some $\rho>0$.
\end{theorem}

\begin{proof}
Without loss of generality, assume that $Y$ is diagonal in the basis $\{e_1,e_2,e_3\}$. One can show that a left-invariant metric $g$ satisfying $X(g)=Y$ must have this property as well; cf.~\cite[Section~4.2]{Buttsworth19}. Existence and uniqueness of such a metric $g$ were established in~\cite[pages~14--18]{IG08}. To prove~\eqref{mon_qtty}, denote 
$$
\rho_i=g(e_i,e_i),\qquad i=1,2,3.
$$
For simplicity, assume that $\rho_1\ge\rho_2\ge\rho_3$. A straightforward computation based on~\cite[Theorem~4.3]{Milnor} shows that
\begin{align*}
X(g)(e_i,e_i)=\frac{(\rho_i^2+\rho_j^2-3\rho_k^2-2\rho_i\rho_j+2\rho_i\rho_k+2\rho_j\rho_k)(\rho_i^2-3\rho_j^2+\rho_k^2+2\rho_i\rho_j-2\rho_i\rho_k
+2\rho_j\rho_k)}{16\rho_i\rho_j^2\rho_k^2}
\end{align*}
for every permutation $(i,j,k)$ of $\{1,2,3\}$; cf.~\cite[pages~14--16]{IG08}. As a consequence,
\begin{align*}
\frac{X(g)(e_1,e_1)}{X(g)(e_3,e_3)}=\frac{\rho_1((\rho_1+\rho_3-\rho_2)^2+4\rho_3(\rho_2-\rho_3))}{\rho_3((\rho_1+\rho_3-\rho_2)^2-4\rho_1(\rho_1-\rho_2))}.
\end{align*}
It becomes clear that
\begin{align*}
\lambda(Y)=\lambda(X(g))\ge\frac{X(g)(e_1,e_1)}{X(g)(e_3,e_3)}\ge\frac{\rho_1}{\rho_3}=\lambda(g)
\end{align*}
with equalities holding throughout if and only if $\rho_1=\rho_2=\rho_3$, i.e., $g=\rho g_0$ for some $\rho>0$.
\end{proof}

Theorem~\ref{thm_SU2} can be used to establish new interesting dynamical properties of the map~${g\mapsto X(g)}$. For example, the following result describes the periodic orbits of this map.

\begin{cor}\label{cor_nstein}
Given $n\in\mathbb N$ and a left-invariant metric $g$ on~$SU(2)$, suppose that
$$
\underbrace{X(X(\cdots X(X}_{j~\mathrm{times}}(g))\cdots))
$$
is positive-definite for every $j=1,\ldots,n$. If
\begin{align*}
\underbrace{X(X(\cdots X(X}_{n~\mathrm{times}}(g))\cdots))=g,
\end{align*}
then $g=\rho g_0$ for some~$\rho>0$.
\end{cor}

Scaling properties of the cross curvature (0,2)-tensor and the equality $X(g_0)=\frac1{16}g_0$ imply that $\rho=\frac14$ in Corollary~\ref{cor_nstein} if $n$ is odd.

\section*{Acknowledgements}

The second-named author had several discussions of the prescribed cross curvature problem with David Hartley. These discussions took place when David was a research associate at the University of Queensland in~2013. The unpublished notes he produced during that time were a useful resource for us at the initial stages of this project. Specifically, they helped us arrive at Proposition~\ref{prop_12to02}, Theorem~\ref{thm_local} and Lemma~\ref{so2equationslemma}. We express our gratitude to David for his contribution.

\end{document}